\documentclass[a4paper, 10pt]{article}
\usepackage{graphicx,grffile}
\usepackage{amsmath, amssymb, amsthm}
\usepackage{caption}
\usepackage{subfig}
\usepackage{diagbox}
\usepackage{xcolor}
\usepackage{grffile}
\usepackage{mathrsfs}
\usepackage{multirow}
\usepackage[normalem]{ulem}
\usepackage{bm}
\usepackage{enumerate}

\usepackage{indentfirst} 
\allowdisplaybreaks[4]

\newtheorem{theorem}{Theorem}
\newtheorem{lemma}[theorem]{Lemma}
\newtheorem{proposition}[theorem]{Proposition}
\newtheorem{definition}[theorem]{Definition}
\numberwithin{theorem}{section}

\newcommand\bx{\boldsymbol{x}}

\newcommand\bbR{\mathbb{R}}

\newcommand\dd{\,\mathrm{d}}

\newcommand\pd[2]{\dfrac{\partial {#1}}{\partial {#2}}}

\newcommand{\sP}{\mathscr{P}}

\numberwithin{equation}{section}

\graphicspath{{images/}}

{\theoremstyle{remark} }

\newcommand\deletei{\bgroup\markoverwith{\textcolor{blue}{\rule[0.5ex]{2pt}{2pt}}}\ULon}

\title{$Q$-tensor gradient flow with quasi-entropy and discretizations preserving physical constraints\footnote{Yanli Wang is partially supported by NSFC No. 12171026,  U1930402 and 12031013. 
Jie Xu is partially supported by NSFC No. 12001524. }}

\author{Yanli Wang\footnote{Beijing Computational Science Research Center, Beijing, China (ylwang@csrc.ac.cn)}, 
Jie Xu\footnote{LSEC and NCMIS, Institute of Computational Mathematics and Scientific/Engineering Computing (ICMSEC), Academy of Mathematics and Systems Science (AMSS), Chinese Academy of Sciences, Beijing, China (xujie@lsec.cc.ac.cn)}
}

\begin{document}
\date{}
\maketitle

\begin{abstract}
We propose and analyze numerical schemes for the gradient flow of $Q$-tensor with the quasi-entropy. 
The quasi-entropy is a strictly convex, rotationally invariant elementary function, giving a singular potential constraining the eigenvalues of $Q$ within the physical range $(-1/3,2/3)$. 
Compared with the potential derived from the Bingham distribution, the quasi-entropy has the same asymptotic behavior and underlying physics. 
Meanwhile, it is very easy to evaluate because of its simple expression. 
For the elastic energy, we include all the rotationally invariant terms. 
The numerical schemes for the gradient flow are built on the nice properties of the quasi-entropy. 
The first-order time discretization is uniquely solvable, keeping the physical constraints and energy dissipation, which are all independent of the time step. 
The second-order time discretization keeps the first two properties unconditionally, and the third with an $O(1)$ restriction on the time step. 
These results also hold when we further incorporate a second-order discretization in space. 
Error estimates are also established for time discretization and full discretization. 
Numerical examples about defect patterns are presented to validate the theoretical results. 

  \vspace*{4mm}
  \end{abstract}

\section{Introduction}
\subsection{Background}
Liquid crystals are a class of matters exhibiting local orientational order.
They are typically formed by rigid, anisotropically-shaped molecules, such as rod-like molecules.
Within a volume infinitesimal but containing many rod-like molecules, the orientation of rod-like molecules might be distributed uniformly or preferring certain direction, which are two different types of orientational order. To classify such different states theoretically, orientational order parameters need to be introduced. For rod-like molecules, the order parameter most commonly chosen is a second-order symmetric traceless tensor $Q$. Let us represent the direction of a rod by a unit vector $\bm{m}$. 
The tensor $Q$ is defined from the second moment of $\bm{m}$ about the normalized orientational density function $\rho(\bm{m})$, 
\begin{align}
  Q=\int_{S^2}(\bm{m}\otimes\bm{m}-\frac{1}{3}I)\rho(\bm{m})\dd\bm{m}, \label{Qdef}
\end{align}
where $I$ is the identity matrix. Armed with the tensor $Q$, one can classify the isotropic and uniaxial nematic phases for spatially homogeneous systems.
Moreover, the elasticity and defects of the nematic phase can be described with a tensor field for spatially inhomogeneous systems. 


To study phase transitions, elasticity and defects, it is necessary to construct some free energy about $Q$. 
The simplest and most widely used one is the Landau-de Gennes free energy \cite{deGennesProst1993}, which is a polynomial about $Q$ and its spatial derivatives consisting of a few rotational invariant terms. 
Usually the polynomial about $Q$, called bulk energy, is truncated at fourth order to describe the isotropic-nematic phase transition; the terms with spatial derivatives, called elastic energy, are truncated at second order. 
Energy minima and saddle points are explored to reveal the locally stable and transition states \cite{Cohen1987,  Alouges1997, Nguyen2013, Golovaty2014,Adler2015,Canevari2017,yin2020construction}. The free energy can be further incorporated into dynamic models \cite{Beris1994, Qian1998, wang2002hydro, Wu2019Dyan, zhao2017novel}.
Dynamic models are energy dissipative PDE systems coupling the equation of $Q$ and the Navier-Stokes equation. 
When the velocity is small and neglected, a dynamic model can be reduced to a gradient flow about $Q$.
Gradient flows about the Landau-de Gennes free energy are also studied extensively \cite{Fukuda2004, Ravnik2009, cai2017stable, iyer2015dynamic}. The $Q$-tensor free energy and dynamics have received much attention, since it is able to describe fascinating defect patterns observed experimentally. Previous works suggest that defect patterns can be greatly affected by the geometry, boundary conditions, coefficients \cite{Canevari2017, Hu2016, Wang2018}, indicating the subtlety of the model. For more details, we refer to the review \cite{Zhanglei2021} and the references therein.

Although a Landau-type free energy is simple and useful, it has a drawback that the order parameter tensor $Q$ cannot be constrained in the physical range. To clarify this range, let us go back to the definition of the tensor $Q$. 
Since $Q+I/3$ is the second moment, it shall be positive definite, or equivalently its eigenvalues shall be positive. 
Together with the fact that the trace of $Q+I/3$ equals one, its eigenvalues shall be less than one. 
Thus, the eigenvalues of $Q$ shall lie within $(-1/3,2/3)$. 
On the other hand, it is impossible for any polynomial about $Q$ to impose such a constraint.
As a result, when the free energy contains some high-order terms, it can become unbounded from below \cite{Ball2010}.
It will also bring ambiguity to interpret the results if $Q$ is outside the physical range.

There are actually few works discussing imposing constraints on the eigenvalues of $Q$. 
One remedy is to include in the free energy a term derived from the Bingham distribution \cite{Ball2010, Han2014from, Song2015, Luo2018} that is the maximum entropy state when $Q$ is given. 
From now on, we call it the Bingham term. With the Bingham term, the bulk energy possesses only one coefficient that could be related to physical parameters, and the isotropic-nematic phase transition can be described clearly \cite{Liu2005Axial, Fatkullin2005}. 
As a side remark, such a term gives a singular function that brings difficulty in PDE analyses, which has raised interests in PDE community \cite{Liu2021Regularity}.

Despite the Bingham term has the advantage of constraining the tensor in the physical range, it is defined through an integral on the unit sphere. This brings notable difficulty in both analysis and computation. Numerically, various approaches for fast computation have been developed \cite{Luo2018, Kent1987, Kume2013, Kume2005, Wang2008Crucial, Grosso2000, Yu2010, Yu2007, jiang2021efficient}, but they turn out to be far from the demand of numerically solving PDEs involving the Bingham term efficiently and accurately. 

In this work, we shall consider a substitution of the Bingham term, called the quasi-entropy that is proposed and analyzed in \cite{Xu2020quasi} for general cases of symmetries and order parameters. 
The quasi-entropy is an elementary function, but maintains the essential properties held by the Bingham term: strict convexity, ability to impose physical constraints, rotational invariance. 
The quasi-entropy for the tensor $Q$ is actually reduced from the general cases. 
Its asymptotic behaviors are consistent with the Bingham term. 
In addition, armed with the quasi-entropy, the isotropic-nematic phase transition can be captured correctly. 
We also point out that the quasi-entropy is more convenient for theories of rigid molecules having complex symmetries. When dealing with these molecules, since multiple tensors are necessarily involved, even a fourth-order polynomial is too complicated. 
Instead, with the quasi-entropy one only needs to include a second-order polynomial in the bulk energy, resulting in an energy much simpler while keeping the underlying physics. 

The problem we consider throughout this paper is a gradient flow of a free energy about $Q$ containing the quasi-entropy. 
Provided that the quasi-entropy has the nice property of imposing the physical constraints, it is highly desirable that they can be preserved in numerical schemes. Our target is to establish and analyze  numerical methods that are able to maintain the physical constraints at the discrete level, while keeping energy dissipation, the significant property of gradient flows. 

\subsection{Numerical approaches}

The target of keeping $Q$ in the physical range seems analogous to proposing positivity/bound preserving schemes for a single unknown function, which have been discussed in numerous works. The available approaches include seeking maximum principles \cite{du2020maximum, liu1996non, CHEN2016198}, cut-off or introducing Lagrange multipliers \cite{Cheng2021, MR3062022, MR4186541}, reconstruction by limiters \cite{Zhang2010Onmax, Zhang2010OnPos}, change of variables \cite{HuaS21, liu2018positivity}, and using barrier functions \cite{ShenXu2021, ShenXu2020, Chen2019Pos}. 
To maintain the $Q$-tensor in the physical range, however, turns out to be a different problem. 
There are five degrees of freedom in $Q$. 
In other words, any PDE about $Q$ can be written as a PDE system about five scalar functions. 
The range of $Q$ imposes constraints on the five scalars, which are strongly coupled between the five scalars. 
%
%
%
%
Generally, to deal with range preserving for multiple variables, difficulties could arise if we would like to utilize the approaches keeping the bounds for a single variable: 
\begin{itemize}
\item Maximum principle. It relies on the property of elliptic operators. However, it is known that elliptic operators on multiple unknown functions, if not decoupled, does not possess maximum principles. 
\item Reconstruction approach. It usually requires that the average is positive or lies within the bounding interval. The average, however, does not seem an appropriate concept for a range on multiple variables. 
\item Cut-off or Lagrange multiplier approach. Similarly, this is a confusing concept in the multiple variable cases. 
\item Change of variables. This approach is possible for multiple variables. 
But it is usually a difficult task to find out such a mapping. 
Even if it is available, it usually leads to a PDE much more complicated and other structures in the original PDE might be destructed. 
\item Barrier (convex) function. This approach is suitable only when such a function can be proposed. 
\end{itemize}

For the $Q$-tensor gradient flow, the quasi-entropy happens to be suitable to act as a barrier function because of its nice properties. 
Thus, we will construct numerical schemes based on this approach. 
Nevertheless, the fact that the quasi-entropy is a singular potential leads to difficulties in the analysis of numerical schemes.

First, we shall discretize the gradient flow in time, for which first- and second-order schemes are constructed. 
For the first-order scheme, we are able to prove the unique solvability in the physical range and energy dissipation unconditionally, i.e. with arbitrary time step. 
For the second-order scheme, we can prove the unique solvability in the physical range unconditionally, while the energy dissipation requires a restriction on the time step. 
Next, we discuss the discretization in space, for which we consider the finite difference schemes. 
The strategy is to discretize the free energy first, then to find its derivatives about the values on the grid points. 
In this way, we arrive at the numerical schemes consistent with the energy dissipation. 
For the fully discretized schemes, we can prove the similar results as the time-discretized schemes. 
Error estimates for both time discretization and full discretization are established. 

The approach is also suitable for PDEs for liquid crystals formed by other types of rigid molecules, such as cone-shaped, rectagular, triangular, bent-core, tetrahedral, etc. 
In this sense, the aim of the current work is to establish some fundamental results on the numerical aspects for PDEs of this class and bring such problems to the community.

\subsection{Outline}
Sec. \ref{sec:gradient-flow} is dedicated to the formulation of the gradient flow, where we shall discuss the quasi-entropy and its properties. 
The numerical schemes and analyses are presented in Sec. \ref{sec:numerical}. 
Numerical examples are given in Sec. \ref{sec:examples} to validate our theoretical results.
In particular, we pay special attention to defect patterns. 
In Sec. \ref{sec:conclusion}, we draw a conclusion and point out directions of future works. 
Some related topics are discussed in Appendix.

\section{Quasi-entropy and gradient flow}\label{sec:gradient-flow}
\subsection{Notations}
We begin with defining some notations about tensors. 
Recall that our order parameter is a second-order symmetric traceless tensor (or a $3\times 3$ matrix) $Q$, i.e.
\begin{align}
  Q=Q^t,\qquad \text{tr}Q=0. \nonumber
\end{align}
Its $(i,j)$ component is denoted by $Q_{ij}$ where $i,j=1,2,3$. 
The symmetric and traceless requirements can also be expressed as $Q_{ij}=Q_{ji}$ and $Q_{ii}=0$. 
Hereafter, we shall adopt the convention of summations on repeated indices, so that the traceless condition $Q_{ii}=0$ shall be understood as $\sum_{i=1}^3Q_{ii}=0$. 
For the indices appearing in the partial derivatives, $\partial_1$, $\partial_2$, $\partial_3$ represent the partial derivatives along the $x$-, $y$-, $z$-direction, respectively. 

The notation of dot product is introduced between two tensors of the same order, for which we explain by a couple of examples.
Suppose $A$ and $B$ are two matrices. Their gradients, $\nabla A$ and $\nabla B$, are third-order tensors. 
The dot product is understood as follows, 
\begin{align}
  A\cdot B=A_{ij}B_{ij},\qquad \nabla A\cdot\nabla B=\partial_kA_{ij}\partial_kB_{ij}. 
\end{align}
For the dot products between identical vectors or matrices, we denote 
\begin{align}
  |A|^2=A\cdot A,\qquad |\nabla A|^2=\nabla A\cdot\nabla B. 
\end{align}
For a region $\Omega$, the inner product and correspondingly the $L^2$ norm are defined as 
\begin{align}
  (A,B)=\int_{\Omega} A\cdot B\dd\bm{r},\qquad \|A\|^2=(A,A). 
\end{align}
The $H^k$ norms can be defined similarly. 

Let us also introduce the Kronecker delta,
\begin{align}
  \delta_{ij}=\left\{
  \begin{array}{ll}
    1, & i=j;\\
    0, & i\ne j.
  \end{array}
  \right.
\end{align}
It is closely related to the identity tensor ($3\times 3$ matrix) $I=\text{diag}(1,1,1)$.

We denote by $\mathcal{Q}$ the space of symmetric traceless tensors, 
\begin{align}
  \mathcal{Q}\triangleq \left\{A\in\bbR^{3\times3}|A_{ij}=A_{ji},\ A_{ii}=0\right\}. 
\end{align}
Correspondingly, we define the projection operator $\sP$ onto the space $\mathcal{Q}$ as 
\begin{align}
  (\sP A)_{ij}=\frac 12 (A_{ij}+A_{ji})-\frac 13 \delta_{ij}A_{kk}. 
\end{align}
Oftentimes, we also use the notation $\sP(A_{ij})$ for the same meaning. 
It can be verified easily that if $S\in\mathcal{Q}$, then for any second-order tensor $A$, it holds  
\begin{align}
  A\cdot S=\sP A\cdot S. \label{proj_invrt}
\end{align}
Moreover, we define the set of physical $Q$-tensors as 
\begin{align}
  \mathcal{Q}_{\text{phys}}\triangleq\left\{A\in\mathcal{Q}|\lambda_i(A)\in\left(-\frac 13,\frac 23\right)\right\}, 
\end{align}
where we use $\lambda_i(A)$ to denote the eigenvalues of $A$. 

\subsection{Quasi-entropy and bulk energy}
As we have mentioned, the quasi-entropy is an elementary function about several angular moment tensors, acting as a substitution of the term derived from the maximum entropy state. 
The quasi-entropy is defined for any choice of angular moment tensors \cite{Xu2020quasi}. 
In this work, we shall introduce the special case where only the tensor $Q$, defined in \eqref{Qdef}, is involved. 

For the tensor $Q$, the quasi-entropy is given by 
\begin{equation}
    q(Q)=-\ln\det\left(Q+\frac{1}{3}I\right)-2\ln\det\left(\frac{1}{3}I-\frac{1}{2}Q\right). \label{qe_Q}
\end{equation}
The domain of the above function is all symmetric traceless $Q$ such that the two tensors $Q+\frac 13 I$ and $\frac 13 I-\frac 12 Q$ are positive definite. It is easy to verify that this domain is exactly $\mathcal{Q}_{\text{phys}}$. 
A significant note is that the coefficient $-2$ before the second log-determinant is essential as it originates from symmetry reduction \cite{Xu2020quasi}. 

We shall show some properties of the quasi-entropy $q(Q)$. The following proposition is actually a special case of general quasi-entropy discussed in \cite{Xu2020quasi}. Since the proof in that work is presented in an abstract manner, we still write down the proof specifically for \eqref{qe_Q}. 
\begin{proposition}\label{qeprop}
  The quasi-entropy $q$ has the following properties:
  \begin{enumerate}
  \item Its domain $\mathcal{Q}_{\text{phys}}$ is bounded and strictly convex, and it is strictly convex on $\mathcal{Q}_{\text{phys}}$. 
  \item It is invariant of rotations: for any $3\times 3$ orthogonal matrix $T$, it holds $q(TQT^t)=q(Q)$. 
  \item It is lower-bounded and gives a barrier function on $\mathcal{Q}_{\text{phys}}$: if any eigenvalues of $Q$ tends to $(-1/3)^+$ or $(2/3)^-$, then $q(Q)$ tends to $+\infty$. 
  \end{enumerate}
\end{proposition}
\begin{proof}
  First, we show the boundedness and convexity of $\mathcal{Q}_{\text{phys}}$.
  For boundedness, note that the diagonal components of $Q+I/3$ and $I/3-Q/2$ are positive, which implies $-1/3<Q_{11},Q_{22},Q_{33}<2/3$.
  Then, for off-diagonal components, because any principle minor of $Q+I/3$ are positive definite, we deduce that   
  \begin{align*}
    Q_{12}^2<\left(Q_{11}+\frac 13\right)\left(Q_{22}+\frac 13\right)<1, 
  \end{align*}
  For convexity, assume that $Q_1,Q_2\in \mathcal{Q}_{\text{phys}}$.
  Since $Q_1+I/3$ and $Q_2+I/3$ are positive definite, their average, $(Q_1+Q_2)/2+I/3$, is also positive definite. 
  Similarly, $I/3-(Q_1+Q_2)/4$ is positive definite.
  Therefore, $(Q_1+Q_2)/2\in\mathcal{Q}_{\text{phys}}$. 
  
  The strict convexity of $q(Q)$ is an immediate result of the fact that log-determinant is concave, which we show below. 
  Assume that $V$ and $V\pm R$ are positive definite. Then we can write $V=LL^T$, where $L$ is invertible. 
  Thus, 
  \begin{align*}
    \log\det (V+R)=&\log\det \big(L(I+L^{-1}RL^{-T})L^T\big)
    =\log\det V + \log\det (I+L^{-1}RL^{-T})\\
    =&\log\det V + \sum_{i=1}^{3}\log(1+\mu_i), 
  \end{align*}
  where $\mu_i$ are the eigenvalues of $L^{-1}RL^{-T}$. 
  Therefore, 
  \begin{equation}
    \log\det (V+R)+\log\det (V-R)-2\log\det V=\sum_{i=1}^{3}\log(1-\mu_i^2)\leqslant 0. \nonumber
  \end{equation}
  The equality holds only when $\mu_i=0$ for $i=1,2,3$, which is equivalent to $R=0$. 

  For rotation invariance, just note that $\det T=1$ and 
  \begin{align*}
    \det\Big(TQT^t+\frac I3\Big)=\det\bigg(T\left(Q+\frac{I}{3} \right)T^t\bigg)=\det(T)\det\Big(Q+\frac I3\Big)\det(T^t)=\det\Big(Q+\frac I3\Big). 
  \end{align*}

  We use the rotational invariance to express the quasi-entropy by the eigenvalues. 
  \begin{align}
    q(Q)=\sum_{i=1}^3 -\ln\left(\frac{1}{3} +\lambda_i(Q)\right)-2\ln\left(\frac{1}{3} - \frac 12 \lambda_i(Q)\right). \label{qent_eig}
  \end{align}
  It follows from the concavity of $\ln s$ that 
  \begin{align*}
    \ln\left(\frac 13 +s\right)+2\ln\left(\frac{1}{3} - \frac{1}{2}s\right)\leqslant 3\ln\frac 13,
  \end{align*}
  so that $q(Q)$ is bounded from below. 
  It remains to show that $q(Q)$ gives a barrier function on $\mathcal{Q}_{\text{phys}}$. 
  Note that an eigenvalue of $Q$ tends to $(2/3)^-$ implies that the other two tend to $(-1/3)^+$, so we only discuss the case $\min_i\lambda_i(Q)\to (-1/3)^+$ and show $q(Q)\to +\infty$.
  By $\lambda_i(Q)<\frac 23$, we have $-\ln\big(\frac 13 +\lambda_i(Q)\big)>0, -\ln\big(\frac 13 - \frac 12 \lambda_i(Q)\big)>0$. Therefore, we could conclude the proof by 
  \begin{align}
    q(Q)>-\ln\left(\frac 13 +\min_i\lambda_i(Q)\right). 
  \end{align}
\end{proof}
From \eqref{qent_eig}, we can see that the asymptotic behavior is also consistent with the Bingham term when $Q$ tends to the boundary of $\mathcal{Q}_{\text{phys}}$. We shall briefly introduce the properties of the Bingham term in Appendix \ref{app:bingham}. 

Because of the third point in the above proposition, we could naturally extend $q(Q)$ to be defined in $\mathcal{Q}$ by allowing it to take $+\infty$: 
\begin{align}
  q(Q)=\left\{
  \begin{array}{ll}
    -\ln\det(Q+\frac{1}{3}I)-2\ln\det\left(\frac{1}{3}I-\frac{1}{2}Q\right), & Q\in\mathcal{Q}_{\text{phys}};\\
    +\infty, & Q\in\mathcal{Q}\backslash\mathcal{Q}_{\text{phys}}.
  \end{array}
  \right.
\end{align}
Using this notation would be convenient in our analysis afterward. 

The bulk energy consists of the quasi-entropy and a quadratic term, given by 
\begin{equation}
    f_b(Q)=q(Q)-\frac{1}{2}c_{02}|Q|^2. \label{Qbulk}
\end{equation}
The coefficient $c_{02}$ can be connected to the interaction intensity between rod-like molecules.
The stationary points of \eqref{Qbulk} are of particular interest.
It is easy to see that the bulk energy is rotationally invariant, 
$$
f_b(Q)=f_b(TQT^t), 
$$
from Proposition \ref{qeprop} and the fact that $|Q|^2$ is rotationally invariant. 
Thus, the stationary points of $f_b$ are also rotationally invariant. 
In other words, if $Q$ is a stationary point, then for any orthogonal matrix $T$, $TQT^t$ is also a stationary point.

The stationary points of \eqref{Qbulk} have been fully classified \cite{Xu2020quasi}, as we summarize below. 
\begin{proposition}\label{uni_stationary}
  The stationary points of $f_b$ have at least two identical eigenvalues, i.e. can be written as 
  \begin{align}
    Q=s(\bm{n}\otimes\bm{n}-I/3), \label{Quni}
  \end{align}
  where the value of $s$ is fixed but $\bm{n}$ can take any unit vector.
\end{proposition}
If we assume that $Q$ has the form \eqref{Quni}, we could write \eqref{Qbulk} as a function of $s$, denoted by $f_b(s)$. 
The stationary points of $f_b(s)$ can be classified as follows. 
\begin{proposition}\label{st-classify}
  There are two critical values of $c_{02}$, denoted by $0<\chi^*<\chi^{**}=27/2$.
  \begin{enumerate}[(i)]
  \item When $c_{02}<\chi^*$, there is only one stationary point $s=0$. 
  \item When $\chi^*<c_{02}<\chi^{**}$, there are three stationary points $0<s_1<s_2$, of which $0$ and $s_2$ are stable, while $s_1$ is unstable. 
  \item When $c_{02}>\chi^{**}$, there are three stationary points $s_1<0<s_2$, of which $s_1$ and $s_2$ are stable, while $0$ is unstable. 
  \end{enumerate}
\end{proposition}
Further numerical studies indicate that in the case (iii), the stationary point $Q=s_1(\bm{n}\otimes\bm{n}-I/3)$ is actually unstable in terms of $Q$, since we now do not require that $Q$ takes the form \eqref{Quni}. 
Proposition \ref{st-classify} is also consistent with the result using the Bingham term (see \cite{Liu2005Axial,Fatkullin2005}). 

In this work, we require $c_{02}>0$. 

\subsection{Elastic energy}
\label{elastic energy}
Since we will investigate the spatially inhomogeneous cases, it is necessary to include the elastic energy.
We shall consider a bounded region $\Omega$ in $\bbR^2$ or $\bbR^3$ with piecewise smooth boundaries. 
For the elastic energy, we include three quadratic terms, given by 
\begin{equation}
  \int_{\Omega}f_e(Q)\dd\bm{r}=\int_{\Omega}\frac{1}{2}\left(c_{21}|\nabla Q|^2+\tilde{c}_{22}\partial_iQ_{ik}\partial_jQ_{jk}+\tilde{c}_{23}\partial_iQ_{jk}\partial_jQ_{ik}\right)\dd\bm{r}. 
\end{equation}

The difference of the last two terms actually only depends on the boundary values of $Q$.
Denote by $\bm{\nu}$ the unit outer normal vector of $\partial\Omega$. 
We could calculate that 
\begin{align}
  \int_{\Omega}\partial_iQ_{jk}\partial_jQ_{ik}-\partial_iQ_{ik}\partial_jQ_{jk}\dd\bm{r}
  =\int_{\partial\Omega}\nu_iQ_{jk}\partial_jQ_{ik}-\nu_jQ_{jk}\partial_iQ_{ik}\dd S.\nonumber
\end{align}
 
For each point on the boundary, we could choose a local orthonormal frame $(\bm{\nu},\bm{u},\bm{v})$. 
Using 
\begin{align}
  \delta_{jl}=\nu_j\nu_l+u_ju_l+v_jv_l,\nonumber
\end{align}
we deduce that 
\begin{align*}
  &\int_{\partial\Omega}\nu_iQ_{jk}\partial_jQ_{ik}-\partial_iQ_{ik}\partial_jQ_{jk}\dd\bm{r}\\
  =&\int_{\partial\Omega}\nu_iQ_{jk}(\nu_j\nu_l+u_ju_l+v_jv_l)\partial_lQ_{ik}-\nu_jQ_{jk}(\nu_i\nu_l+u_iu_l+v_iv_l)\partial_lQ_{ik}\dd S\\
  =&\int_{\partial\Omega}\nu_iQ_{jk}(u_ju_l+v_jv_l)\partial_lQ_{ik}-\nu_jQ_{jk}(u_iu_l+v_iv_l)\partial_lQ_{ik}\dd S.
\end{align*}
Since $\bm{u}$ and $\bm{v}$ are tangent to $\partial\Omega$, the derivatives $u_l\partial_lQ$ and $v_l\partial_lQ$ depend only on the boundary values. 
Therefore, if we consider the Dirichlet boundary condition 
\begin{align}
  Q|_{\partial\Omega}=Q_{\text{bnd}}, \label{Qbnd}
\end{align}
we could define $c_{22}=\tilde{c}_{22}+\tilde{c}_{23}$ and write the elastic energy as 
\begin{equation}
  \int_{\Omega}f_e(Q) \dd \bm{r} =\int_{\Omega}\frac{1}{2}\left(c_{21}|\nabla Q|^2+c_{22}\partial_iQ_{ik}\partial_jQ_{jk}\right)\dd \bm{r}, \label{Qelastic}
\end{equation}
where we ignore the constant given by the surface integral because it makes no difference in the free energy. 

We require that the elastic energy be positive definite, which imposes conditions on the coefficients. Since we have 
\begin{align}
  \partial_iQ_{jk}\partial_iQ_{jk}-\partial_iQ_{jk}\partial_jQ_{ik}
  =\frac 12 (\partial_iQ_{jk}-\partial_jQ_{ik})(\partial_iQ_{jk}-\partial_jQ_{ik})\geqslant 0, 
\end{align}
it suffices to require
\begin{align}
  c_{21}>0,\qquad c_{21}+c_{22}>0. \label{elasc}
\end{align}
Indeed, if $c_{22}\geqslant 0$, we use the term $\partial_iQ_{ik}\partial_jQ_{jk}\geqslant 0$.
If $-c_{21}<c_{22}<0$, we use the above equality to write the elastic energy as 
\begin{align}
  \int_{\Omega}\frac 12(c_{21}+c_{22})\partial_iQ_{jk}\partial_iQ_{jk}-\frac 14 c_{22}(\partial_iQ_{jk}-\partial_jQ_{ik})(\partial_iQ_{jk}-\partial_jQ_{ik})\dd\bm{r}\geqslant 0. \label{elas2}
\end{align}

\subsection{Gradient flow}
Summarizing \eqref{Qbulk} and \eqref{Qelastic}, the total free energy is 
\begin{equation}
\label{eq:energy_def}
    E[Q]=\int_{\Omega} f_b[Q] + f_e[Q]\dd\bm{r}. 
\end{equation}
It is bounded from below because of Proposition \ref{qeprop} and \eqref{elasc}. 
Moreover, we allow the free energy to take $+\infty$. 
Using the fact that for any nonsingular matrix $R$, it holds 
\begin{align}
  \pd{\ln\det R}{R}=R^{-t}, 
\end{align}
we deduce (formally) the variational derivative as 
\begin{align}
  \left(\frac{\delta E}{\delta Q}\right)_{ij}\triangleq&\sP\left(-\left(Q+\frac I3\right)^{-1}_{ij}+\left(\frac I3-\frac Q2\right)^{-1}_{ij}-c_{02}Q_{ij}-c_{21}\Delta Q_{ij}-c_{22}\partial_{ik}Q_{jk}\right)\nonumber\\
  =&-\left(Q+\frac I3\right)^{-1}_{ij}+\left(\frac I3-\frac Q2\right)^{-1}_{ij}+\frac 13\Bigg(\left(Q+\frac I3\right)^{-1}_{kk}-\left(\frac I3-\frac Q2\right)^{-1}_{kk}\Bigg)\delta_{ij}-c_{02}Q_{ij}\nonumber\\
  &-c_{21}\Delta Q_{ij}
  -\frac 12 c_{22}\big(\partial_{ik}Q_{jk}+\partial_{jk}Q_{ik}-\frac 23 \partial_{kl}Q_{kl}\delta_{ij}\big), \label{varE}
\end{align}
where we recall that $\sP$ is the projection operator to the space $\mathcal{Q}$. 
Since the quasi-entropy is a singular potential, rigorously the variational derivative \eqref{varE} shall be understood as subdifferential (cf. \cite{Liu2021Regularity}). Nevertheless, we will use the notation of variational derivatives throughout the rest of the paper.

The gradient flow is written as 
\begin{equation}
  \frac{\partial Q}{\partial t}=-\frac{\delta E}{\delta Q},
\end{equation}
together with the Dirichlet boundary condition \eqref{Qbnd}. 
For the boundary condition, it shall satisfy that there exists a $Q^{(0)}\in H^1(\Omega;\mathcal{Q}_{\text{phys}})$, such that
\begin{align}
  Q^{(0)}|_{\partial\Omega}=Q_{\text{bnd}},\qquad E[Q^{(0)}]<+\infty. \label{bndcon}
\end{align}
The initial condition is given by 
\begin{align}
  Q(\bm{r},0)=Q_{\text{ini}}(\bm{r})\in L^2(\Omega).
\end{align}
The gradient flow satisfies the energy law
\begin{align*}
  E[Q(\bm{r},T)]-E[Q(\bm{r},0)]=-\frac 12\int_0^T \|\partial_tQ\|^2+\left\|\frac{\delta E}{\delta Q}\right\|^2\dd t. 
\end{align*}
\textit{Remark.} Due to the singular potential, the above energy law is not quite obvious. Its derivation requires the theory of gradient flows on a metric space, which we do not discuss in the current work and refer \cite{ambrosio2008gradient,Liu2021Regularity} for details. 

The $c_{22}$ term makes the elliptic operator in \eqref{varE} much more complicated than elliptic operators about a scalar function. 
For example, the maximum principle of any kind no longer holds. 
The interplay of the singular potential and the elliptic operator might lead to subtle behavior of the solution, as indicated by the analysis in \cite{Liu2021Regularity}.
In particular, the solution might be arbitrarily close to the boundary of $\mathcal{Q}_{\text{phys}}$. 
Thus, when considering numerical schemes, we also need to take special care on constraining the solution in $\mathcal{Q}_{\text{phys}}$. 

\section{Numerical methods}
\label{sec:numerical}
Our main target is to propose a numerical scheme keeping the solution within $\mathcal{Q}_{\text{phys}}$. Meanwhile, the elliptic operator cannot bring any help for this target since no maximum principle is guaranteed. Thus, we utilize the property of the quasi-entropy. We first discretize the gradient flow in time, followed by discretizing it in space. 

Before we begin our discussion, we state a couple of inequalities about the quasi-entropy due to its convexity. 
\begin{lemma}
\label{qent_ineq}
  For the quasi-entropy $q(Q)$, it holds
  \begin{align}
  q(S_2)-q(S_1)\geqslant \frac{\partial q}{\partial S_1}\cdot (S_2-S_1)=\sP\frac{\partial q}{\partial S_1}\cdot (S_2-S_1), \\
  \left(\frac{\partial q}{\partial S_1}-\frac{\partial q}{\partial S_2}\right)\cdot (S_1-S_2)=\sP\left(\frac{\partial q}{\partial S_1}-\frac{\partial q}{\partial S_2}\right)\cdot (S_1-S_2)\geqslant 0, 
  \end{align}
  where $S_1,S_2\in\mathcal{Q}_{\text{phys}}$. The equality holds only when $S_1=S_2$. 
\end{lemma}
\begin{proof}
  Since $q$ is strictly convex, the function 
  \begin{align}
    h(t)=q(tS_2+(1-t)S_1)
  \end{align}
  is convex about $t$, and is strictly convex if $S_1\ne S_2$. Its derivative is calculated as 
  \begin{align}
    h'(t)=\left.\frac{\partial q}{\partial Q}\right|_{Q=tS_2+(1-t)S_1}\cdot (S_2-S_1). 
  \end{align}
  By $h(1)-h(0)\geqslant h'(0)$ and $h'(1)\geqslant h'(0)$,  we obtain the two inequalities.
  When $S_1\ne S_2$, the strict inequalities hold because $h$ is strictly convex. 
\end{proof}
In addition, let us state the definition of subdifferential for a strictly convex functional. 
\begin{definition}
\label{subdiff}
  Let $F[\phi]$ be a strictly convex functional about $\phi$. The subdifferential of $F$ at $\phi_0$ is a set consisting of all $\psi$ that satisfies: for arbitrary $\phi_1$, it holds 
  \begin{equation}
    F[\phi_1]-F[\phi_0]\geqslant (\psi,\phi_1-\phi_0). \nonumber
  \end{equation}
\end{definition}
Note that it is possible that the subdifferential is an empty set. 

\subsection{Time discretization}
\subsubsection{First-order scheme}
We begin with a first-order scheme, 
\begin{align}
\label{tfirst}
  \frac{Q_{ij}^{n+1}-Q_{ij}^n}{\delta t}=&-\sP\left(\frac{\partial q}{\partial Q^{n+1}}\right)_{ij}+c_{02}Q_{ij}^n+c_{21}\Delta Q_{ij}^{n+1}\nonumber\\
  &+\frac{c_{22}}{2}\big(\partial_{ik}Q^{n+1}_{kj}+\partial_{jk}Q^{n+1}_{ki}-\frac 23\partial_{kl}Q^{n+1}_{kl}\delta_{ij}\big),\nonumber\\
  Q^{n+1}|_{\partial \Omega}=&\,Q_{\text{bnd}}. 
\end{align}
In the above, we discretize the quasi-entropy implicitly, leading to a nonlinear scheme. We shall prove that the scheme indeed possesses the properties we desire. 
\begin{theorem}\label{tfirst_thm}
  For arbitrary $\delta t$, the scheme \eqref{tfirst} has a unique solution in $H^1(\Omega)$ satisfying $Q^{n+1}(\bm{r})\in \mathcal{Q}_{\text{phys}}$ almost everywhere in $\Omega$ and $E[Q^{n+1}]<+\infty$. Moreover, it satisfies the energy law, 
  \begin{align}
    E[Q^{n+1}]+\frac{1+c_{02}\delta t}{2\delta t}\|Q^{n+1}-Q^n\|^2+\int_{\Omega}f_e(Q^{n+1}-Q^n)\dd\bm{r}\leqslant E[Q^n]. \label{enfirst}
  \end{align}
\end{theorem}
\begin{proof}
  We could reformulate the scheme as the unique minimizer of a strictly convex functional. 

  Consider the functional 
  \begin{align}
    F[Q^{n+1}]=\frac{1}{2\delta t}\|Q^{n+1}-Q^n\|^2+\int_{\Omega} q(Q^{n+1})+f_e(Q^{n+1})-c_{02}Q^n\cdot Q^{n+1}\dd\bm{r}. \label{En_scheme}
  \end{align}
  This functional is strictly convex about $Q^{n+1}$ by looking into each term:
  \begin{itemize}
  \item It is obvious that $\frac{1}{2\delta t}\|Q^{n+1}-Q^n\|^2$ is convex.
  \item The quasi-entropy term $\int_{\Omega} q(Q^{n+1})\dd\bm{r}$ is strictly convex. 
  \item By \eqref{elasc}, the quadratic elastic term $\int_{\Omega}f_e(Q^{n+1})\dd\bm{r}$ is convex. 
  \item The last term $-c_{02}(Q^n,Q^{n+1})$ is linear, thus convex. 
  \end{itemize}
  Furthermore, the first three terms are bounded from below, and $Q^{n+1}\in L^{\infty}$ can be controlled by the quadratic terms. 
  Therefore, $F[Q^{n+1}]$ is bounded from below. 
  Next, we show that there exists a unique minimizer of \eqref{En_scheme}. Recall that we have assumed \eqref{bndcon}. Since 
  \begin{align*}
    F[Q^{n+1}]-E[Q^{n+1}]=&\frac{1}{2\delta t}\|Q^{n+1}-Q^n\|^2
    +\frac 12 c_{02}\|Q^{n+1}\|^2 -c_{02}(Q^n,Q^{n+1})\\
    \leqslant &C(\|Q^{n}\|^2+\|Q^{n+1}\|^2). 
  \end{align*}
  Thus, we set $Q^{n+1}=Q^{(0)}$ to arrive at $F[Q^{(0)}]<+\infty$, so that $-\infty<\inf F[Q^{n+1}]<+\infty$. 
  Therefore, let us choose a sequence $Q^{(k)}$ such that 
  \begin{align*}
    \lim_{k\to+\infty}F[Q^{(k)}]=\inf F[Q^{n+1}]. 
  \end{align*}
  By \eqref{elasc}, it is clear that the functional $F$ can control the $H^1$ norm, so that the sequence $Q^{(k)}$ is $H^1$ bounded. Therefore, we could choose a subsequence that converges $H^1$-weakly to $\bar{Q}$. This subsequence satisfies 
  \begin{align}
    \lim_{k\to+\infty}F_1[Q^{(k)}]=\inf F_1[\bar{Q}]. \label{quad_cont}
  \end{align}
  where
  \begin{align*}
    F_1[Q^{n+1}]=\frac{1}{2\delta t}\|Q^{n+1}-Q^n\|^2+\int_{\Omega} f_e(Q^{n+1})-c_{02}Q^n\cdot Q^{n+1}\dd\bm{r}. 
  \end{align*}
  Then, using Riesz Theorem, we could further choose an a.e. convergent subsequence. Since $q$ is bounded from below, we use Fatou's lemma to derive that
  \begin{align}
    \int_{\Omega}q(\bar{Q})\dd\bm{r}\leqslant \liminf_{k\to+\infty}\int_{\Omega}q(Q^{(k)})\dd\bm{r}. \label{qFatou}
  \end{align}
  Combining \eqref{quad_cont} and \eqref{qFatou}, we arrive at 
  \begin{align}
    \inf F[Q^{n+1}]\leqslant F[\bar{Q}]\leqslant \liminf_{k\to+\infty}F_1[Q^{(k)}]+\lim_{k\to+\infty}\int_{\Omega}q(Q^{(k)}) \dd \bm{r}\leqslant \inf F[Q^{n+1}], 
  \end{align}
  which implies that $\bar{Q}$ is a minimizer of $F[Q^{n+1}]$.
  Its uniqueness follows directly from strict convexity. 
  
  We verify that $\bar{Q}$ is a.e. in $\mathcal{Q}_{\text{phys}}$. 
  From $F[Q^{(k)}]<+\infty$, we deduce that $\int_{\Omega}q(Q^{(k)})<+\infty$.
  It implies that each $Q^{(k)}$ lies within $\mathcal{Q}_{\text{phys}}$ a.e. in $\Omega$. 
  Therefore, the limiting function $\bar{Q}$ satisfies $\lambda(\bar{Q})\in [-1/3,2/3]$ a.e. in $\Omega$.
  The measure of the set where $\lambda(\bar{Q})=1/3$ or $-2/3$ must be zero. 
  Otherwise, we deduce that $F[\bar{Q}]\geqslant \int_{\Omega}q(\bar{Q}) \dd \bm{r} -C=+\infty$, which is a contradiction. 

  It remains to show that the minimizer is a solution to the scheme.
  At the unique global minimizer, the zero function is a subdifferential by definition. 
  We could use similar arguments of \cite[Lemma 3.3]{Liu2021Regularity} to claim that: when subdifferential exists, it only has one element that equals to the left-hand side of \eqref{tfirst} minus its right-hand side. 
  Thus, the difference of both-hand sides of \eqref{tfirst} must be zero, so that the unique minimizer gives a solution to the scheme. 
  
  Although $F[Q^{n+1}]$ has a unique minimizer, it does not necessarily imply that the solution to the scheme \eqref{tfirst} is unique. This is because it is possible that there is a solution where the subdifferential of the functional $F$ is an empty set. 
  For this reason, we still need to show the uniqueness of the solution to \eqref{tfirst}. 
  Suppose that we have two solutions $Q^{n+1,(1)}$ and $Q^{n+1,(2)}$. Denote $S=Q^{n+1,(1)}-Q^{n+1,(2)}$. Then we have 
  \begin{align}
    \frac{S_{ij}}{\delta t}=&-\sP\left(\frac{\partial q}{\partial Q^{n+1,(1)}}-\frac{\partial q}{\partial Q^{n+1,(2)}}\right)_{ij}+c_{21}\Delta S_{ij}^{n+1}\nonumber\\
    &+\frac{c_{22}}{2}\big(\partial_{ik}S^{n+1}_{kj}+\partial_{jk}S^{n+1}_{ki}-\frac 23\partial_{kl}S^{n+1}_{kl}\delta_{ij}\big), \label{Sunique}\\
    S|_{\partial \Omega}=&0. \nonumber
  \end{align}
  Taking the inner product of \eqref{Sunique} with $S$ and noting \eqref{proj_invrt}, we obtain 
  \begin{align}
    \frac{1}{\delta t}\|S\|^2+\int_{\Omega}\sP\left(\frac{\partial q}{\partial Q^{n+1,(1)}}-\frac{\partial q}{\partial Q^{n+1,(2)}}\right)\cdot (Q^{n+1,(1)}-Q^{n+1,(2)}) +2f_e(S)\dd\bm{r}=0. \label{Seng}
  \end{align}
  Since $S$ is zero on $\partial \Omega$, it is not difficult to check that $\int_{\Omega}f_e(S)\dd\bm{r}\geqslant 0$. Moreover, since both $Q^{n+1,(1)}$ and $Q^{n+1,(2)}$ are a.e. within $\mathcal{Q}_{\text{phys}}$, we use Lemma \ref{qent_ineq} to obtain that 
  \begin{align}
    \sP\left(\frac{\partial q}{\partial Q^{n+1,(1)}}-\frac{\partial q}{\partial Q^{n+1,(2)}}\right)\cdot (Q^{n+1,(1)}-Q^{n+1,(2)})\geqslant 0,\quad \text{a.e. in }\Omega. 
  \end{align}
  Thus, the left-hand side of \eqref{Seng} is nonnegative, and takes zero only when $S=0$ a.e. in $\Omega$. 
  
  Now we turn to energy dissipation. The case $E[Q^n]=+\infty$ is obvious. We assume $E[Q^n]<+\infty$, so that $Q^n\in \mathcal{Q}_{\text{phys}}$ a.e. in $\Omega$. 
  Taking inner product of \eqref{tfirst} with $Q^{n+1}-Q^n$, we arrive at 
  \begin{align}
    \frac{\|Q^{n+1}-Q^n\|^2}{\delta t}=&-\left(\sP\Big(\frac{\partial q}{\partial Q^{n+1}}\Big),Q^{n+1}-Q^n\right)+c_{02}(Q^n,Q^{n+1}-Q^n)\nonumber\\
    &+c_{21}(\Delta Q^{n+1},Q^{n+1}-Q^n)\nonumber\\
    &+c_{22}\int_{\Omega}\partial_{ik}Q^{n+1}_{kj}(Q_{ij}^{n+1}-Q_{ij}^n)\dd\bm{r}. \label{enfirst1}
  \end{align}
  Using Lemma \ref{qent_ineq} we have
  \begin{align*}
    -\sP\Big(\frac{\partial q}{\partial Q^{n+1}}\Big)\cdot(Q^{n+1}-Q^n)\leqslant -q(Q^{n+1})+q(Q^n),\quad \text{a.e. in }\Omega. 
  \end{align*}
  For the other terms on the right-hand side, some direct calculations along with $Q^{n+1}-Q^n=0$ on $\partial \Omega$ yield 
  \begin{align*}
    (Q^n,Q^{n+1}-Q^n)=&\frac 12 (\|Q^{n+1}\|^2-\|Q^n\|^2-\|Q^{n+1}-Q^n\|^2), \\
    (\Delta Q^{n+1},Q^{n+1}-Q^n)=&-\Big(\nabla Q^{n+1},\nabla(Q^{n+1}-Q^n)\Big)\nonumber\\
    =&\frac 12 (-\|\nabla Q^{n+1}\|^2+\|\nabla Q^n\|^2-\|\nabla (Q^{n+1}-Q^n)\|^2), \\
    \int_{\Omega}\partial_{ik}Q^{n+1}_{kj}(Q_{ij}^{n+1}-Q_{ij}^n)\dd\bm{r}=&
    -\int_{\Omega}\partial_jQ^{n+1}_{jk}\partial_i(Q_{ik}^{n+1}-Q_{ik}^n)\dd\bm{r}\nonumber\\
    =&\frac 12\int_{\Omega}-\partial_iQ^{n+1}_{ik}\partial_jQ^{n+1}_{jk}+\partial_iQ^n_{ik}\partial_jQ^n_{jk}\nonumber\\
    &\qquad -\partial_i(Q_{ik}^{n+1}-Q_{ik}^n)\partial_j(Q_{jk}^{n+1}-Q_{jk}^n)\dd\bm{r}. 
  \end{align*}
  Taking the inequalities above into \eqref{enfirst1}, we arrive at \eqref{enfirst}. 
\end{proof}

We then carry out the error analysis of the scheme \eqref{tfirst}. Define $R^n(\bm{r})=Q^n(\bm{r})-Q(\bm{r},t^n)$. 
\begin{theorem}\label{errfirst_thm}
  Assume that $\partial_tQ, \partial_{tt}Q\in L^2\big(0,T;L^2(\Omega)\big)$. For the scheme \eqref{tfirst}, the following error estimate holds, 
\begin{equation}
    \|R^n\|^2\leqslant C\exp\Big((1-C\delta t)^{-1}t^n\Big)\delta t^2\int_0^{t^n}\|\partial_tQ\|^2+\|\partial_{tt}Q\|^2\dd t, \label{errfirst}
\end{equation}
where the constant $C\sim c_{02}$. 
\end{theorem}
\begin{proof}
  We deduce the equation for $R^n$, 
  \begin{align}
    \frac{R_{ij}^{n+1}-R_{ij}^n}{\delta t}+T_1^n = &-\sP\left(\frac{\partial q}{\partial Q^{n+1}(\bm{r})}-\frac{\partial q}{\partial Q(\bm{r},t^{n+1})}\right)_{ij}+c_{02}R_{ij}^n+c_{21}\Delta R_{ij}^{n+1}\nonumber\\
    &+\frac{c_{22}}{2}\left(\partial_{ik}R^{n+1}_{kj}+\partial_{jk}R^{n+1}_{ki}-\frac{2}{3}\partial_{kl}R^{n+1}_{kl}\delta_{ij}\right)+c_{02}T_2^n. \label{Rfirst}
  \end{align}
  The truncation errors are given by 
  \begin{align}
    T_1^n = &\frac{1}{\delta t}\Big(Q(\bm{r},t^{n+1})-Q(\bm{r},t^n)\Big)-\partial_tQ(\bm{r},t^{n+1})=\frac{1}{\delta t}\int_{t^n}^{t^{n+1}}(t^n-s)\partial_{tt}Q\dd s,\nonumber\\
    T_2^n = &Q(\bm{r},t^{n+1})-Q(\bm{r},t^n)=\int_{t^n}^{t^{n+1}}\partial_tQ\dd s. \label{t-trunc}
  \end{align}
  Take inner product of \eqref{Rfirst} with $R^{n+1}$.
  Notice that 
  \begin{equation}
    \sP\Big(-\frac{\partial q}{\partial Q^{n+1}(\bm{r})}+\frac{\partial q}{\partial Q(\bm{r},t^{n+1})}\Big)\cdot \Big(Q^{n+1}(\bm{r})-Q(\bm{r},t^{n+1})\Big)\leqslant 0. \nonumber
  \end{equation}
  Moreover, since $R^{n+1}|_{\partial\Omega}=0$, we deduce that 
  \begin{align*}
    &\int_{\Omega}R_{ij}^{n+1}\bigg(c_{21}\Delta R_{ij}^{n+1}+\frac{c_{22}}{2}\left(\partial_{ik}R^{n+1}_{kj}+\partial_{jk}R^{n+1}_{ki}-\frac{2}{3}\partial_{kl}R^{n+1}_{kl}\delta_{ij}\right)\bigg)\dd\bm{r}\nonumber\\
    =&-\int_{\Omega}f_e(R^{n+1})\dd\bm{r}\leqslant 0. 
  \end{align*}
  Therefore, we could derive that 
  \begin{align*}
    \frac{1}{2\delta t}(\|R^{n+1}\|^2-\|R^n\|^2)\leqslant (c_{02}+1)(\|R^{n+1}\|^2+\|R^n\|^2)+(c_{02}+1)(\|T_1^n\|^2+\|T_2^n\|^2). 
  \end{align*}
  The truncation errors can be estimated as 
  \begin{align*}
    \|T_1^n\|^2\leqslant C\delta t\int_{t^n}^{t^{n+1}}\|\partial_{tt}Q\|^2\dd s,\qquad
    \|T_2^n\|^2\leqslant C\delta t\int_{t^n}^{t^{n+1}}\|\partial_tQ\|^2\dd s.
  \end{align*}
  Using Gronwall's inequality (see, for example, \cite{Quarter2008}), we deduce \eqref{errfirst}. 
\end{proof}

\subsubsection{Second-order scheme}
We could build a second-order scheme based on BDF2 as 
\begin{align}
\label{tsecond}
  \frac{3Q_{ij}^{n+1}-4Q_{ij}^n+Q_{ij}^{n-1}}{2\delta t}=&-\sP\left(\frac{\partial q}{\partial Q^{n+1}}\right)_{ij}+c_{02}(2Q_{ij}^n-Q_{ij}^{n-1})+c_{21}\Delta Q_{ij}^{n+1}\nonumber\\
  &+\frac{c_{22}}{2}\left(\partial_{ik}Q^{n+1}_{kj}+\partial_{jk}Q^{n+1}_{ki}-\frac 23\partial_{kl}Q^{n+1}_{kl}\delta_{ij}\right). 
\end{align}
Following the similar routine in the last section, we could prove the theorems below. 
\begin{theorem}\label{tsecond_thm}
  For arbitrary $\delta t$, the scheme \eqref{tsecond} has a unique solution in $H^1(\Omega)$ satisfying $Q^{n+1}(\bm{r})\in \mathcal{Q}_{\text{phys}}$ a.e. in $\Omega$ and $E[Q^{n+1}]<+\infty$. Furthermore, if $c_{02}\delta t\leqslant 2$, the following energy dissipation holds, 
  \begin{align}
    E[Q^{n+1}]+\frac{1+2c_{02}\delta t}{4\delta t}\|Q^{n+1}-Q^n\|^2+\int_{\Omega}f_e(Q^{n+1}-Q^n)\dd\bm{r}\nonumber\\
    \leqslant E[Q^n]+\frac{1+2c_{02}\delta t}{4\delta t}\|Q^n-Q^{n-1}\|^2. \label{ensecond}
  \end{align}
\end{theorem}
\begin{proof}
  The proof is similar to Theorem \ref{tfirst_thm}. We simply point out the differences.

  For the existence and uniqueness, we also formulate the solution to the scheme as the unique minimizer of a strictly convex functional. The functional is now given by 
  \begin{align}
    F[Q^{n+1}]=&\frac{3}{4\delta t}\|Q^{n+1}\|^2+\frac{1}{2\delta t}(-4Q^n+Q^{n-1},Q^{n+1})\nonumber\\
    &+\int_{\Omega} q(Q^{n+1})+f_e(Q^{n+1})-c_{02}(2Q^n-Q^{n-1})\cdot Q^{n+1}\dd\bm{r}. 
  \end{align}
  Follow the same route of Theorem \ref{tfirst_thm} to arrive at existence and Uniqueness. 
  
  Taking the inner product of \eqref{tsecond} with $Q^{n+1}-Q^n$, we deduce that 
  \begin{align}
    &\frac{1}{2\delta t}(3Q^{n+1}-4Q^n+Q^{n-1},Q^{n+1}-Q^n)\nonumber\\
    =
    &c_{02}(2Q^n-Q^{n-1},Q^{n+1}-Q^n)\nonumber\\
    &-\left(\sP\Big(\frac{\partial q}{\partial Q^{n+1}}\Big),Q^{n+1}-Q^n\right)\nonumber\\
    &+c_{21}\left(\Delta Q^{n+1},Q^{n+1}-Q^n\right)\nonumber\\
    &+c_{22}\int_{\Omega}\partial_{ik}Q^{n+1}_{kj}(Q_{ij}^{n+1}-Q_{ij}^n)\dd\bm{r}. \label{ensecond1}
  \end{align}
  The last three lines are dealt with in the same way as in Theorem \ref{tfirst_thm}.
  For the first two lines, we use the equalities below, 
  \begin{align*}
    (2Q^n-Q^{n-1},Q^{n+1}-Q^n)=&\frac 12 (\|Q^{n+1}\|^2-\|Q^n\|^2)\nonumber\\
    &-\frac 12\|Q^{n+1}-Q^n\|^2+(Q^{n+1}-Q^n,Q^n-Q^{n-1}),\\
    (3Q^{n+1}-4Q^n+Q^{n-1},Q^{n+1}-Q^n)=&\frac 32\|Q^{n+1}-Q^n\|^2-\frac 12(Q^{n+1}-Q^n,Q^n-Q^{n-1}). 
  \end{align*}
  Thus, we arrive at 
  \begin{align*}
    &E[Q^{n+1}]+\int_{\Omega}f_e(Q^{n+1}-Q^n)\dd\bm{r}\nonumber\\
    &\qquad +\frac{3+c_{02}\delta t}{2\delta t}\|Q^{n+1}-Q^n\|^2-\frac{1+2c_{02}\delta t}{2\delta t}(Q^{n+1}-Q^n,Q^n-Q^{n-1})
    \leqslant E[Q^n]. 
  \end{align*}
  When $c_{02}\delta t\leqslant 2$, we have $3+c_{02}\delta t\geqslant 1+2c_{02}\delta t$.
  Use 
  \begin{align*}
    &\|Q^{n+1}-Q^n\|^2-(Q^{n+1}-Q^n,Q^n-Q^{n-1})\nonumber\\
    =&\frac 12 \left(\|Q^{n+1}-Q^n\|^2-\|Q^n-Q^{n-1}\|^2+\|Q^{n+1}-2Q^n+Q^{n-1}\|^2\right)
  \end{align*}
  to obtain \eqref{ensecond}. 
\end{proof}
Similar to Theorem \ref{errfirst_thm}, we have the following error estimate. 
\begin{theorem}\label{errsecond_thm}
  Assume that $\partial_{tt}Q, \partial_{ttt}Q\in L^2\big(0,T;L^2(\Omega)\big)$. For the scheme \eqref{tsecond}, the following error estimate holds, 
\begin{equation}
  \|R^n\|^2+\|2R^n-R^{n-1}\|^2\leqslant C\exp\Big((1-C\delta t)^{-1}t^n\Big)\delta t^4\int_0^{t^n}\|\partial_{tt}Q\|^2+\|\partial_{ttt}Q\|^2\dd t, \label{errsecond}
\end{equation}
where the constant $C\sim c_{02}$. 
\end{theorem}

\subsection{Full discretization}
The properties of time discretizations can be inherited by full discretizations if we carefully keep the integration by parts. 
Let us consider a square region in 2D, $\Omega=[0,L]^2$. 
The strategy is to discretize the free energy in space, then take derivatives about the values on the discretized points to arrive at the scheme.

We divide $[0,L]^2$ into $N^2$ cells that are small squares of the same size. 
For $0\leqslant l, m\leqslant N$, the index $(l,m)$ represents the point $(lh,mh)$ where $h=L/N$. 
Since the bulk energy can be discretized trivially, we focus on the elastic energy. 
Note that the elastic energy is quadratic about first-order spatial derivatives. 
Thus, we shall start from discretizing the first-order derivatives on each cell. 
In the cell $\Gamma$ whose lower left index is $(l,m)$, the discretization of $\partial_i u, i = 1, 2$  for a scalar function $u$ is given by
\begin{equation}
    \label{eq:spatial_der} 
    \begin{aligned}
\partial_1 u  \approx D_1 u & = \frac{1}{h}\left(-\frac{1}{2}u_{l,m}-\frac{1}{2}u_{l,m+1}+\frac{1}{2}u_{l+1,m}+\frac{1}{2}u_{l+1,m+1}\right), \\
\partial_2 u \approx D_2 u & = \frac{1}{h}\left(-\frac{1}{2}u_{l,m}+\frac{1}{2}u_{l,m+1}-\frac{1}{2}u_{l+1,m}+\frac{1}{2}u_{l+1,m+1}\right).
    \end{aligned}
\end{equation}
When considering the gradient flow in 2D, we actually need to deal with two types of derivatives: $\partial_{11}$ ($\partial_{22}$ is done in the same way) and $\partial_{12}$. 
They actually originate from two different types of terms in the free energy, which we discuss below. 
\begin{itemize}
\item For the term of the form $\partial_1 u \partial_1 v$ in the free energy, its variational derivative about $u$ gives $-\partial_{11}v$. 
Based on \eqref{eq:spatial_der}, the approximation of $\partial_1 u \partial_1 v$ in the cell $\Gamma$ is given by 
\begin{equation}
    \label{eq:spatial_x1x1}
      \begin{aligned}
D_1 u D_1 v = \frac{1}{h^2}&\left(-\frac{1}{2}u_{l,m}-\frac{1}{2}u_{l,m+1}+\frac{1}{2}u_{l+1,m}+\frac{1}{2}u_{l+1,m+1}\right)\\
    &\left(-\frac{1}{2}v_{l,m}-\frac{1}{2}v_{l,m+1}+\frac{1}{2}v_{l+1,m}+\frac{1}{2}v_{l+1,m+1}\right).
  \end{aligned}
\end{equation}
The spatial discretization of $-\partial_{11}v$ at an interior node $(l,m)\;(1\leqslant l,m\leqslant N-1)$ is given by taking the derivatives about $u_{l,m}$ in the discretized free energy. 
Because the node $(l,m)$ is involved in four cells, we need to sum up the derivatives in these cells. 
If we denote the discretized second-order derivative as $D_{11}$, we deduce that at the node $(l,m)$, 
 \begin{equation}
     \label{eq:spatial_xx}
       \begin{aligned}
    -D_{11} v =   &\frac{1}{h^2}\Big[v_{l,m}+\frac{1}{2}v_{l,m+1}+\frac{1}{2}v_{l,m-1}\\
    &-\frac{1}{2}(v_{l+1,m}+\frac{1}{2}v_{l+1,m+1}+\frac{1}{2}v_{l+1,m-1})\\
    &-\frac{1}{2}(v_{l-1,m}+\frac{1}{2}v_{l-1,m+1}+\frac{1}{2}v_{l-1,m-1})\Big]. 
  \end{aligned}
 \end{equation}
By the Taylor expansion, the truncation error has the estimate 
\begin{equation}
    \label{eq:spatial_xx_trun1}
    |\partial_{11} v-D_{11}v| \leqslant \frac{h^2}{3}\max_{\Gamma}\max_{i_1,i_2,i_3,i_4=1,2}|\partial_{i_1i_2i_3i_4}v|. 
\end{equation}

If $u=v$, it is the term $\frac 12 \partial_1u\partial_1u$ that generates $-\partial_{11}u$, and the derivation above still holds. 

\item For the term $\partial_1u\partial_2v$, its variational derivative about $u$ is $-\partial_{12}v$. The discretization of $\partial_1 u \partial_2 v$ on $\Gamma$ reads 
\begin{equation}
    \label{eq:spatial_x1y2}
     \begin{aligned}
D_1 uD_2 v =  \frac{1}{h^2}&\left(-\frac{1}{2}u_{l,m}-\frac{1}{2}u_{l,m+1}+\frac{1}{2}u_{l+1,m}+\frac{1}{2}u_{l+1,m+1}\right)\\
    &\left(-\frac{1}{2}v_{l,m}+\frac{1}{2}v_{l,m+1}-\frac{1}{2}v_{l+1,m}+\frac{1}{2}v_{l+1,m+1}\right).
      \end{aligned}
\end{equation}
Similarly, we collect the derivatives about $u_{l,m}$ in the four cells and obtain the approximation of $-\partial_{12}v$, 
\begin{equation}
    \label{eq:spatial_xy}
    \begin{aligned}
    -D_{12}v=  -\frac{1}{4}v_{l+1,m+1}+\frac{1}{4}v_{l+1,m-1}+\frac{1}{4}v_{l-1,m+1}-\frac{1}{4}v_{l-1,m-1}.
    \end{aligned} 
\end{equation}
The truncation error can be estimated as 
\begin{equation}
    |\partial_{12}v-D_{12}v|\leqslant \frac{2h^2}{3}\max_{\Gamma}\max_{i_1,i_2,i_3,i_4=1,2}|\partial_{i_1i_2i_3i_4}v|. \label{eq:spatial_xx_trun2}
\end{equation}
\end{itemize}

For any term $a$ defined on nodes, we denote by $\sum_P a_P$ the summation over all interior nodes. 
For any term $b$ defined in cells, we denote by $\sum_\Gamma b_{\Gamma}$ the summation of $b$ over all cells. 
Let us write down the equations about summation by parts that we will use later. 
\begin{itemize}
\item For any $u$ and $v$, 
\begin{equation}
    \sum_{\Gamma}(D_1uD_2v-D_1vD_2u)_{\Gamma} \label{disbnd}
\end{equation}
only depends on the values of $u$ and $v$ on the boundary nodes. 
\item If $u$ takes zero on all boundary nodes ($l=0,N$ or $m=0,N$), then we have 
\begin{subequations}
\label{eq:partial_sum}
\begin{align}
   & -\sum_{P} (u D_{11}v)_{P}=\sum_{\Gamma} (D_1 uD_1 v)_{\Gamma}, 
    \label{eq:partial_sum1}  \\
 & -\sum_{P} (u D_{12}v)_{P}=\sum_{\Gamma} (D_2 uD_1 v)_{\Gamma} =\sum_{\Gamma} (D_1 uD_2 v)_{\Gamma},
\label{eq:partial_sum2}
\end{align}
\end{subequations}
\end{itemize}
The derivation of the above equations is straightforward. 
As an example, we show the first equality in \eqref{eq:partial_sum2} in Appendix \ref{app:partial_sum}.

We are now ready to define the discrete energy. 
For the bulk energy, in each cell it is defined from the average of $f_b$ at the four nodes multiplied by $h^2$. 
When summing up over all the cells, because of the Dirichlet boundary conditions, we arrive at 
\begin{align}
    \widehat{E}_b[Q]=h^2\sum_P f_b(Q)_P+A_1, 
\end{align}
where $A_1$ is a constant determined by the value of $Q$ on the boundary nodes. 
For the elastic energy, we define
\begin{align}
\label{eq:pos_E_e}
    \widehat{E}_e[Q]= \frac{h^2}{2}\sum_{\Gamma} c_{21}(D_sQ_{ij}D_sQ_{ij})_{\Gamma}+c_{22}(D_iQ_{ij}D_sQ_{sj})_{\Gamma}.
\end{align}
Using \eqref{disbnd}, we have 
\begin{align}
    \widehat{E}_e[Q]=\frac{h^2}{2}\sum_{\Gamma} c_{21}(D_sQ_{ij}D_sQ_{ij})_{\Gamma}+c_{22}(D_sQ_{ij}D_iQ_{sj})_{\Gamma}+A_2,
     \label{eq:pos_E_e2}
\end{align}
where $A_2$ is a constant determined by the value of $Q$ on the boundary nodes, and $A_2=0$ if $Q$ is zero on all boundary nodes. 
Using the same arguments below \eqref{elasc}, we deduce that the discrete elastic energy $\widehat{E}_e$ is bounded from below, uniformly about $h$. 

The discrete total energy is then defined as  
\begin{equation}
    \label{eq:dis_energy_def}
    \widehat{E}[Q] = \widehat{E}_b[Q]+\widehat{E}_e[Q].
\end{equation}

Define
\begin{equation}
    \label{eq:der_L}
    L_h[Q]_{ij} = -c_{21} D_{kk} Q_{ij} -  c_{22}D_{ik}Q_{kj}.
\end{equation}
The fully discretized schemes of the first and second order are 
\begin{subequations}
\label{eq:full_shceme}
    \begin{align}
      \frac{Q_{ij}^{n+1}-Q_{ij}^n}{\delta t}=& -\sP\left(\frac{\partial q}{\partial Q^{n+1}}\right)_{ij}+c_{02}Q_{ij}^n
     - \sP\left( L_h[Q^{n+1}]\right)_{ij}, 
     \label{eq:full_shcheme_fisrt}  \\
      \frac{3Q_{ij}^{n+1}-4Q_{ij}^n+Q_{ij}^{n-1}}{2\delta t}=&-\sP\left(\frac{\partial q}{\partial Q^{n+1}}\right)_{ij}+c_{02}(2Q_{ij}^n-Q_{ij}^{n-1})- \sP\left(L_h[Q^{n+1}] \right)_{ij}.
  \label{eq:full_shcme_second}
    \end{align}
\end{subequations}
These schemes satisfy the same properties as the corresponding time discretizations, and the proof is also similar. 
Below, we only discuss the first-order scheme, for which we focus on the difference from the time discretization.

\begin{theorem}
\label{thm:dis_energy}
For arbitrary $\delta t$ and $h$, the scheme \eqref{eq:full_shcheme_fisrt} has a unique solution on the domain $\mathcal{Q}_{\rm phys}$. Moreover, the discrete energy law holds, 
\begin{equation}
    \label{eq:dis_energy_dissp}
    \widehat{E}[Q^{n+1}] +h^2\left( \frac{1 + c_{02}\delta t}{2\delta t} \sum_{P}(Q^{n+1}_P - Q^n_P)^2 + \sum_{\Gamma} f_e(Q^{n+1} - Q^n)_{\Gamma}\right)
    \leqslant \widehat{E}[Q^{n}]. 
\end{equation}
\end{theorem}

\begin{proof}
The proof is similar to Theorem \ref{tfirst_thm}, for which we point out the differences. We still reformulate the scheme as the unique minimizer of a strictly convex function. 
Now the function reads 
\begin{equation}
    \label{eq:fun_f}
    \widehat{F}[Q^{n+1}] = h^2\left(\frac{1}{2 \delta t}\sum_{P}(Q^{n+1} - Q^n)_P^2 + \sum_{P}q(Q^{n+1})_{P} + \sum_{\Gamma} f_e(Q^{n+1})_{\Gamma} - c_{02} \sum_{P}(Q^{n} \cdot Q^{n+1})_P\right).
\end{equation}
By Proposition \ref{qeprop} and the discussion above on the elastic energy, we could verify that \eqref{eq:fun_f} is strictly convex on $\mathcal{Q}_{\rm phys}$ and bounded from below. 
Moreover, when $Q$ at any grid point tends to the boundary of $\mathcal{Q}_{\rm phys}$, Proposition \ref{qeprop} implies that $\widehat{F}\to+\infty$. 
Therefore, the function $\widehat{F}[Q^{n+1}]$ has unique minimizer satisfying 
\begin{equation}
    \sP\left(\frac{\partial \widehat{F}}{\partial (Q^{n+1})_P}\right)=0 \label{EL-dis}
\end{equation}
for any interior node $P$. 
Since our spatial discretization is constructed by taking derivatives about the discrete energy, we have 
\begin{equation}
    \frac{\partial}{\partial (Q^{n+1})_P}\sum_{\Gamma} f_e(Q^{n+1})_{\Gamma}=L_h[Q]_P. 
\end{equation}
Thus, we find that \eqref{EL-dis} is exactly \eqref{eq:full_shcheme_fisrt}.

For the energy dissipation, we take dot product of \eqref{eq:full_shcheme_fisrt} with $Q^{n+1} -  Q^n$ and take the sum $\sum_P$. Notice that by \eqref{eq:partial_sum},  the following summation by parts holds, 
\begin{align*}
    &- \sum_P \Big((Q^{n+1} - Q^n) \cdot\sP\left( L_h[Q^{n+1}]\right)\Big)_P\\
    &\quad=\sum_{\Gamma} c_{21}\Big(D_s(Q^{n+1}-Q^n)_{ij}D_sQ^{n+1}_{ij}\Big)_{\Gamma}+c_{22}\Big(D_i(Q^{n+1}-Q^n)_{ij}D_sQ^{n+1}_{sj}\Big)_{\Gamma}.
\end{align*}
Hence, the derivation in Theorem \ref{tfirst_thm} can be followed completely to arrive at \eqref{eq:dis_energy_dissp}. 
\end{proof}

Recall that $R^n(\bm{r})=Q^n(\bm{r})-Q(\bm{r},t^n)$. 
We define the error for the full discretization as 
\begin{equation}\label{error-full}
    \mathcal{E}\triangleq \sqrt{h^2\sum_{P} \left|R^n\right|^2_P}.
\end{equation}

\begin{theorem}
\label{thm:error_esti}
Assume that $Q(\bm{r},t)\in C^4(\Omega\times [0,T])$. The error of the scheme \eqref{eq:full_shcheme_fisrt} has the estimate
  \begin{align}
  \label{eq:dis_error_esti_def}
    \mathcal{E}^2=h^2\sum_{P} \left|R^n\right|^2_P  \leqslant C(\delta t^2+h^4). 
  \end{align}\
  The constant $C$ depends on $T$, the coefficients $c_{02}$, $c_{21}$, $c_{22}$, and the maximum derivatives (up to fourth order) of $Q$ in $\Omega\times [0,T]$. 
\end{theorem}
\begin{proof}
We deduce the equation for $R^n$ as 
\begin{equation}
\label{eq:error_esti}
    \frac{R^{n+1}_{ij} - R^{n}_{ij}}{\delta t} 
    = -\sP\left(\dfrac{\partial q}{\partial Q^{n+1} } - \dfrac{\partial q}{\partial Q(t^{n+1}) } \right)_{ij} + c_{02}R^n_{ij} -\sP\left(L_h[R^{n+1}]\right)_{ij} +T_h^n,
\end{equation}
where the truncation error can be estimated from \eqref{t-trunc}, \eqref{eq:spatial_xx_trun1} and \eqref{eq:spatial_xx_trun2} as 
\begin{equation*}
    |T_h^n|\leqslant C(\delta t+h^2), 
\end{equation*}
where the constant $C$ depends on the maximum derivatives on $\Omega\times [0,T]$. 
Taking dot product of \eqref{eq:error_esti} with $R^{n+1}$ and summing up, 
by \eqref{eq:partial_sum} we deduce that 
\begin{equation*}
    \begin{aligned}
    \sum_{P}(R^{n+1} \cdot \sP L_h[R^{n+1}])_P = \widehat{E}_e[R^{n+1}] \geqslant 0.
    \end{aligned}
\end{equation*}
Follow the same derivation of Theorem \ref{errfirst_thm} to arrive at \eqref{eq:error_esti}. 
\end{proof}

\textit{Remark.} It is noted that the analyses above could be carried out for the Bingham term as well. However, the resulting schemes are not practical at all, since it needs to calculate the derivative of the Bingham term about $Q$, which requires calculating integrals that bring greater difficulties. See Appendix  \ref{app:bingham} for more details.

\section{Numerical examples}
\label{sec:examples}

Let us consider the case where the solution is homogeneous in the $z$-direction, so that we reduce the problem to 2D. 
In our numerical examples, the gradient flow is solved in the square $[0,1]\times [0,1]$. 

We choose $c_{02}>\chi^{**}=13.5$. By Proposition \ref{uni_stationary}, the minimum of the bulk energy $f_b$ is given by $Q_0=s_2(\bm{n}\otimes\bm{n}-I/3)$, where $s_2$ is a function of $c_{02}$ but $\bm{n}$ could take arbitrary unit vector. 
We set the boundary values as the bulk energy minimum, 
\begin{equation}
\label{eq:ini}
    Q=s_2(\bm{n}(\bx)\otimes\bm{n}(\bx)-I/3), 
\end{equation}
where $\bm{n}(\bx)$ depends on the location. 

The numerical schemes are solved using the Newton's iteration. 
To ensure that the iteration lies within $\mathcal{Q}_{\text{phys}}$, we adopt a simple damping strategy: if a Newton's step drives the iteration out of $\mathcal{Q}_{\text{phys}}$, we halve the step length along the Newton's direction, possibly for several times until the iteration falls within $\mathcal{Q}_{\text{phys}}$. 

To illustrate the numerical results, we investigate the principal eigenvector of $Q$. 
Moreover, we define the biaxiality that reflects the features of the eigenvalues (see \cite{Mkaddem2000}). 
If $Q$ is nonzero, the biaxiality is defined as 
\begin{align}
\label{eq:biaaxes}
    1-6\frac{(\mathrm{tr}(Q^3))^2}{(\mathrm{tr}(Q^2))^3}. 
\end{align}
The biaxiality ranges in $[0,1]$. 
It takes zero if and only if $Q$ has two identical eigenvalues, 
and takes one if and only if $Q$ has three eigenvalues $\lambda$, $0$ and $-\lambda$.

\subsection{Accuracy test}
\label{sec:acc}

We first validate the accuracy of the schemes. 
We choose the coefficients 
\begin{equation}
    \label{eq:ex1_ini}
    c_{02} = 20, \qquad c_{21} = 6, \qquad c_{22} = 2.
\end{equation}
The initial condition is chosen as a variation from a steady state, 
\begin{equation}
    \label{ex1_ini}
    Q = Q_0 + \epsilon Q_1(x, y), \qquad (x, y) \in [0, 1] \times [0, 1],
\end{equation}
where the constant tensor $Q_0$ takes the form \eqref{eq:ini} with $\bm{n} = (1, 0, 0)^t$, $\epsilon = 0.05$, and the variation $Q_1(x, y)$ is given by 
\begin{equation}
    \label{eq:ex1_Q1}
    Q_1 = \sin(2\pi x) \sin(2 \pi y). 
\end{equation}
It is noticed that $Q_1$ is zero on the boundary, so that the boundary conditions are set as the constant vector $Q_0$. 
The reference solution at $t=0.01$ is calculated using the second-order scheme with $N=64$ and $\delta t=3.125\times 10^{-5}$. 

Let us begin with the time accuracy. 
We fix $N=64$ 
and choose the time step as $\delta t = 2\times 10^{-3},\, 10^{-3},\, 5\times 10^{-4},\, 2.5\times 10^{-4}$. 
The numerical error $\mathcal{E}$ for the first-order and second-order schemes are plotted in Figure \ref{fig:ex1_first_dt_l2} and \ref{fig:ex1_second_dt_l2}, respectively. We can see clearly the first-order and second-order convergence.

\begin{figure}[!htb]
  \centering
   \subfloat[first-order scheme]{\includegraphics[width=0.45\textwidth,clip]{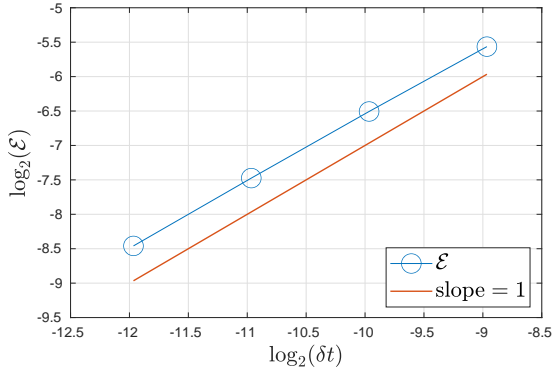}
   \label{fig:ex1_first_dt_l2}
   }\hfill
  \subfloat[second-order scheme]{\includegraphics[width=0.45\textwidth,clip]{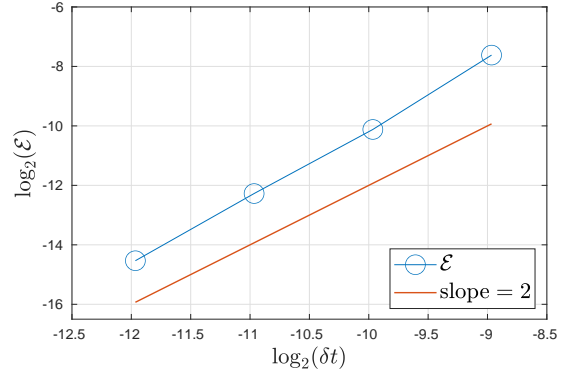}
  \label{fig:ex1_second_dt_l2}
  }
  \caption{ The error $\mathcal{E}$ of numerical solutions to \eqref{eq:full_shcheme_fisrt}, \eqref{eq:full_shcme_second}, with different $\delta t$.  } \label{fig:ex1_dt_l2}
\end{figure}

We turn to the spatial accuracy. Since the time accuracy has been validated, we could match the order in time and the order in space by imposing certain relations between the time step and the grid size. In \eqref{eq:full_shcheme_fisrt}, we consider 
\begin{equation}
    \label{eq:ex1_first_dt}
    \delta t = 0.004 h^2, 
\end{equation}
while in \eqref{eq:full_shcme_second} we consider 
\begin{equation}
    \label{eq:ex1_second_dt}
    \delta t = 0.004 h. 
\end{equation}
Then, we let the number of grid points vary as $N = 2, 4, 8, 16$, and $32$. 
The error $\mathcal{E}$ defined in \eqref{error-full} between the numerical solution and the same reference solution is plotted in Figure \ref{fig:ex1_dx_l2}. 
The second-order convergence is observed, which is consistent with theoretical results.

\begin{figure}[!htb]
  \centering
   \subfloat[first-order scheme]{\includegraphics[width=0.45\textwidth]{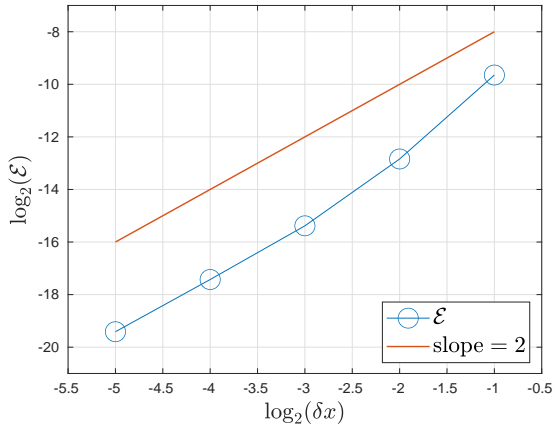}
   \label{fig:ex1_first_dx_l2}
   }\hfill
  \subfloat[second-order scheme]{\includegraphics[width=0.45\textwidth]{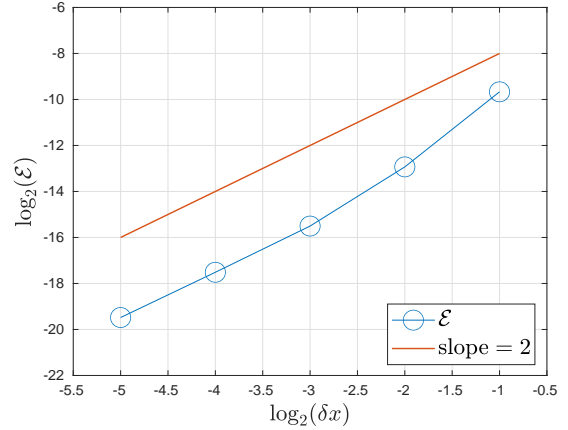}
  \label{fig:ex1_second_dx_l2}
  }
  \caption{The error $\mathcal{E}$ of the numerical solutions to \eqref{eq:full_shcheme_fisrt} and \eqref{eq:full_shcme_second}, with the time steps and grid sizes related by \eqref{eq:ex1_first_dt} and \eqref{eq:ex1_second_dt}, respectively.} 
  \label{fig:ex1_dx_l2}
\end{figure}

\subsection{Solutions with large $c_{02}$}\label{Large-c02}
When $c_{02}$ is large, the scalar $s_2$ minimizing the bulk energy becomes close to one, so that the eigenvalues  are close to $-1/3$ and $2/3$. This is also the case for the solution to the gradient flow, which leads to difficulties numerically. 
Thus, let us test our numerical schemes when the solution is close to the boundary of $\mathcal{Q}_{\text{phys}}$. 

We choose $c_{02}=100$ while keeping $c_{21}=6$ and $c_{22}=2$. 
The initial value is still set according to \eqref{ex1_ini} with $Q_1$ given by \eqref{eq:ex1_Q1} and $\epsilon = 0.001$. 
The boundary condition is set by \eqref{eq:ini} with 
\begin{equation}
    \label{eq:ex2_bou}
    \bm{n}(\bx) = \left\{
    \begin{array}{ll}
     (1, 0, 0)^t,    & x = 0, ~{\rm or}~ x = 1, \\
     
     (0, 1, 0)^t,    & y = 0, ~{\rm or}~ y = 1.
    \end{array}
    \right.
\end{equation}
We use the second-order scheme with $N = 24$ and the time step $\delta t = 0.005$, to evolve the gradient flow till $t = 0.5$ when the system has reached the steady state. 
The maximum and minimum eigenvalues of the matrix $Q$ all over the region, and their distance from $2/3$ or $-1/3$, with respect to the time, are presented in Figure \ref{fig:ex2_eigen_value} and  \ref{fig:ex2_eigen_error}, respectively. 
The eigenvalues are indeed close to the boundary, but still lie within $(-1/3,2/3)$. 

\begin{figure}[!htb]
  \centering
   \subfloat[Maximum and minimum eigenvalues]{\includegraphics[width=0.45\textwidth,height=0.3\textwidth]{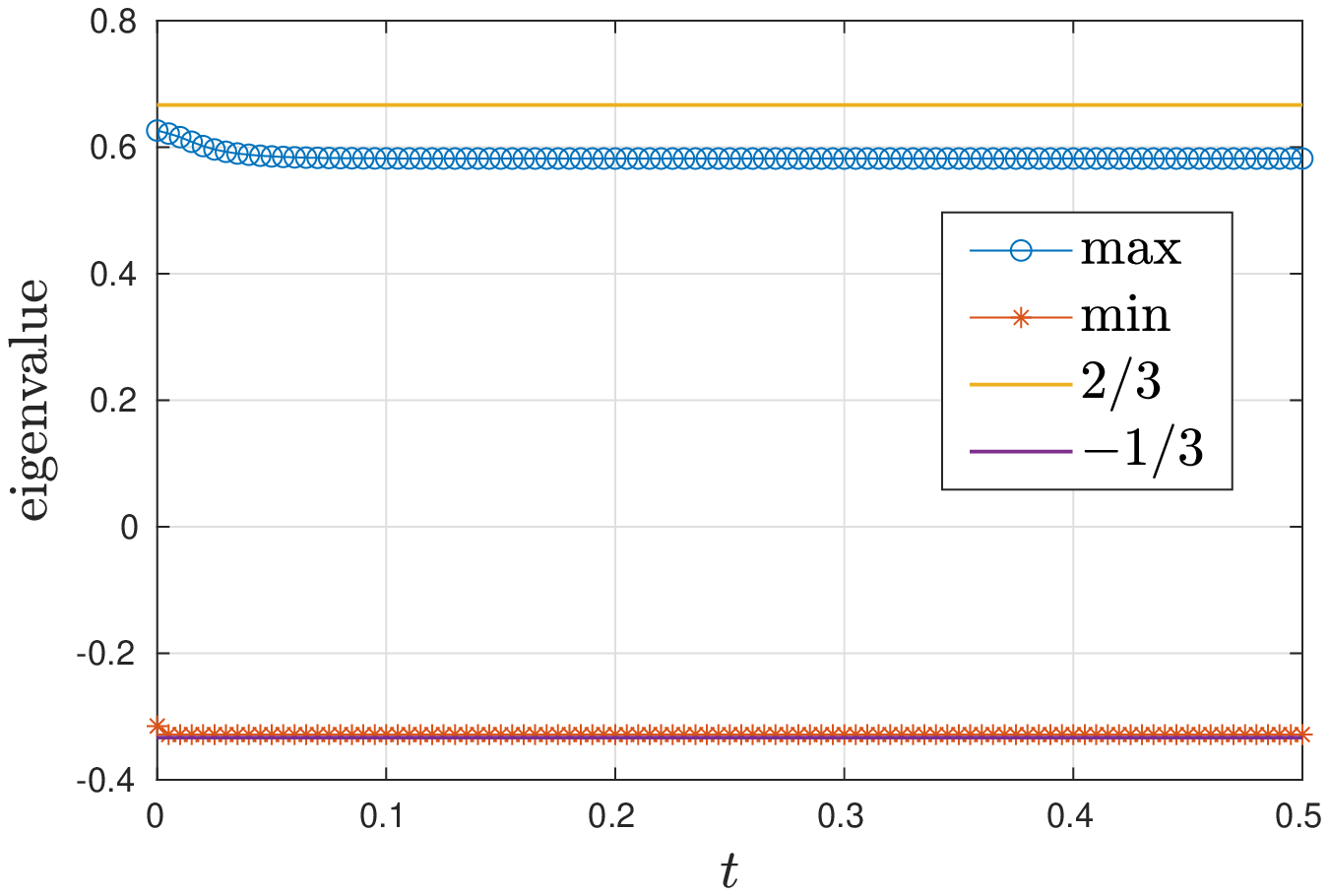}
   \label{fig:ex2_eigen_value}
   }\hfill
  \subfloat[Distance to $2/3$ for the maximum eigenvalue, and to $-1/3$ for the minimum eigenvalue. ]
  {\includegraphics[width=0.45\textwidth,height=0.3\textwidth]{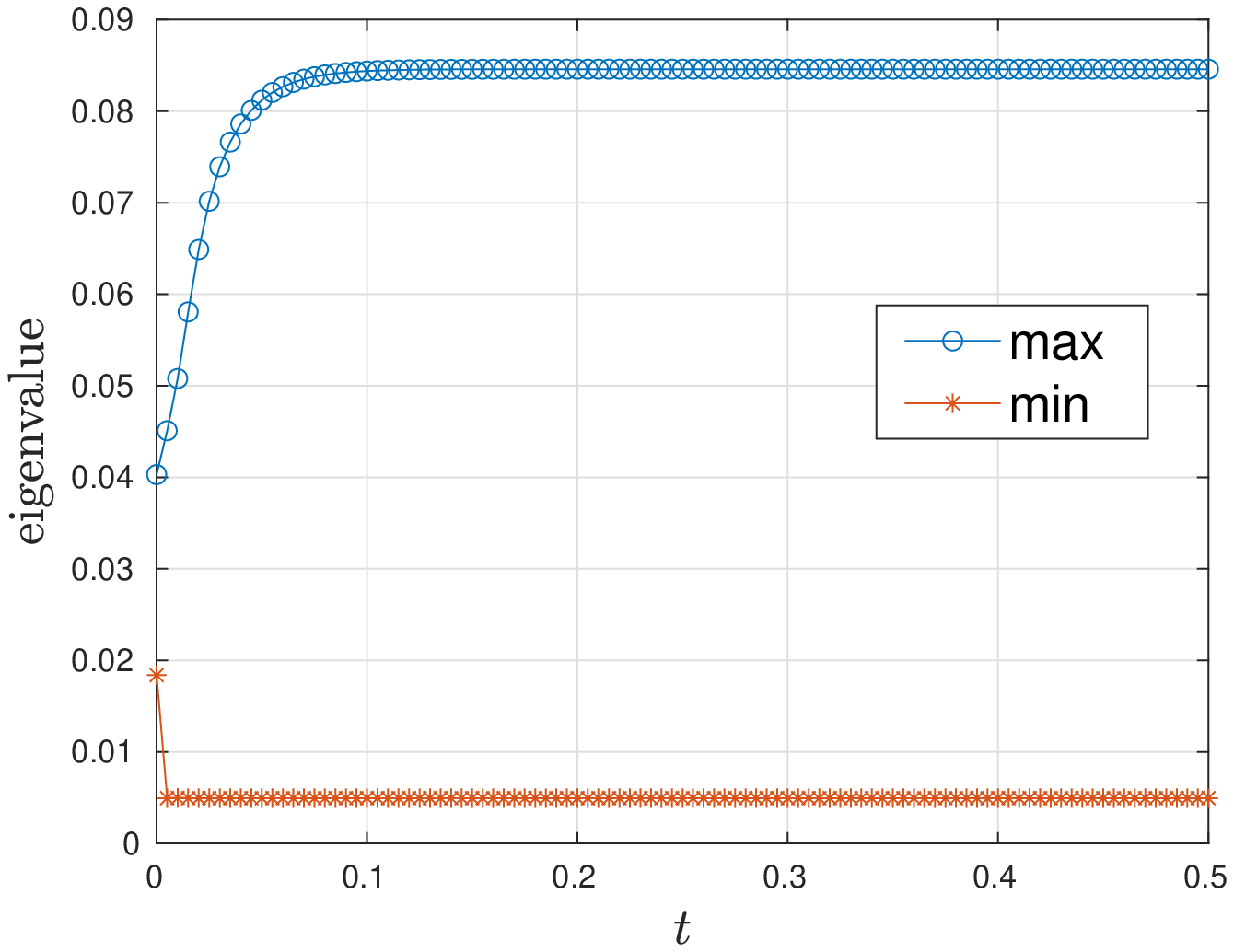}
  \label{fig:ex2_eigen_error}
  }
  \caption{(Section \ref{Large-c02}) The maximum and minimum eigenvalues of the tensor $Q$ all over the region and their minimum distance to $2/3$ or $-1/3$ about the time. } 
  \label{fig:ex2_eig}
\end{figure}

The steady state pattern is shown in Figure \ref{fig:ex2_eig_bia}. 
From the principal eigenvector, we find that the whole region is divided into four subregions by two diagonals. 
Within each of the four subregions, the principal eigenvector is identical to that on the boundary, thus either horizontal or vertical. 
Biaxiality emerges in the middle of each subregion. 
This pattern is called the well order-reconstruction solution in \cite{yin2020construction}.

\begin{figure}[!htb]
  \centering
  \subfloat[Principal eigenvector]{\includegraphics[width=0.45\textwidth, height=0.35\textwidth]{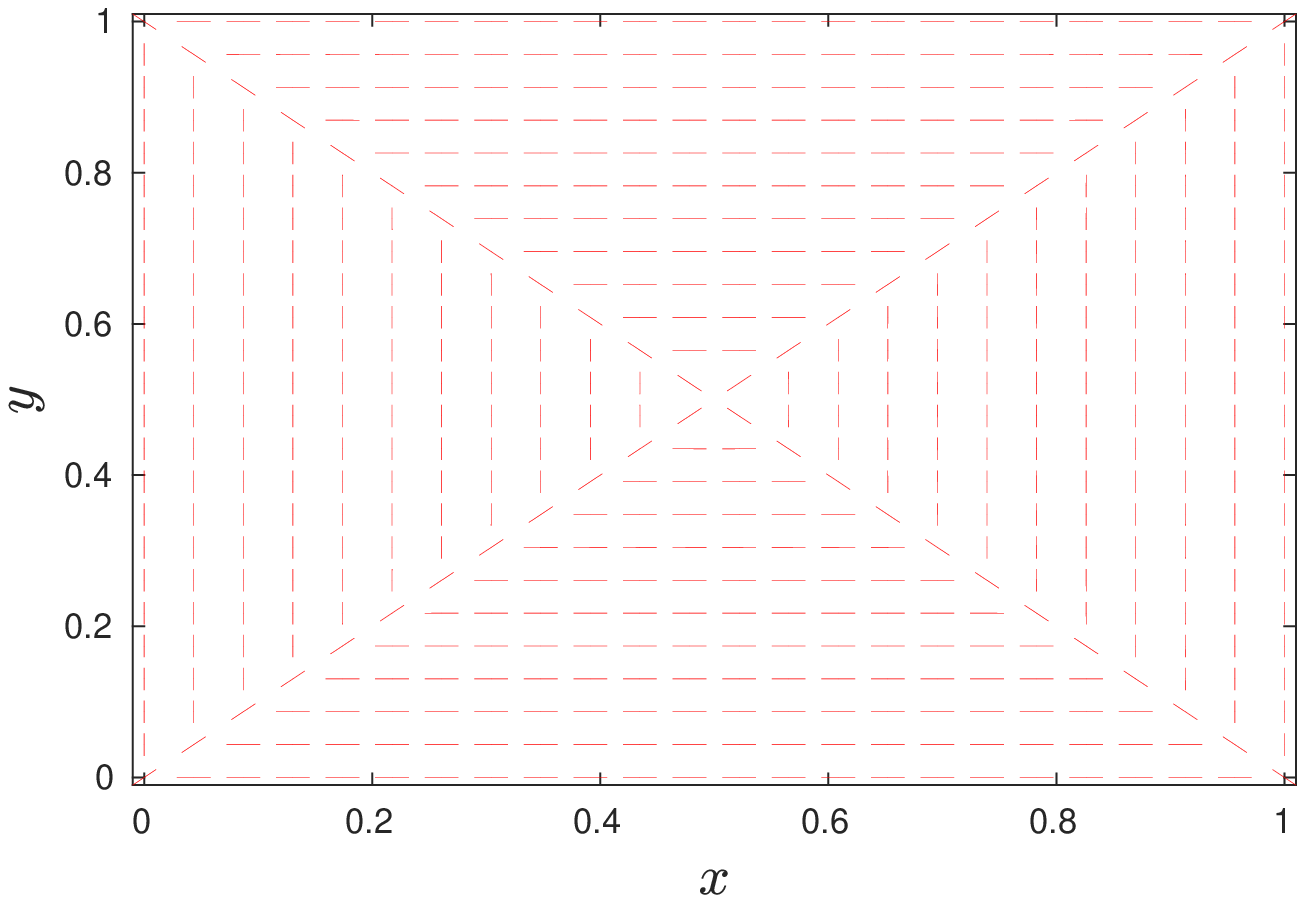}
    \label{fig:ex2_vec}}
\hfill
  \subfloat[Biaxiality]{\includegraphics[width=0.45\textwidth, height=0.35\textwidth]{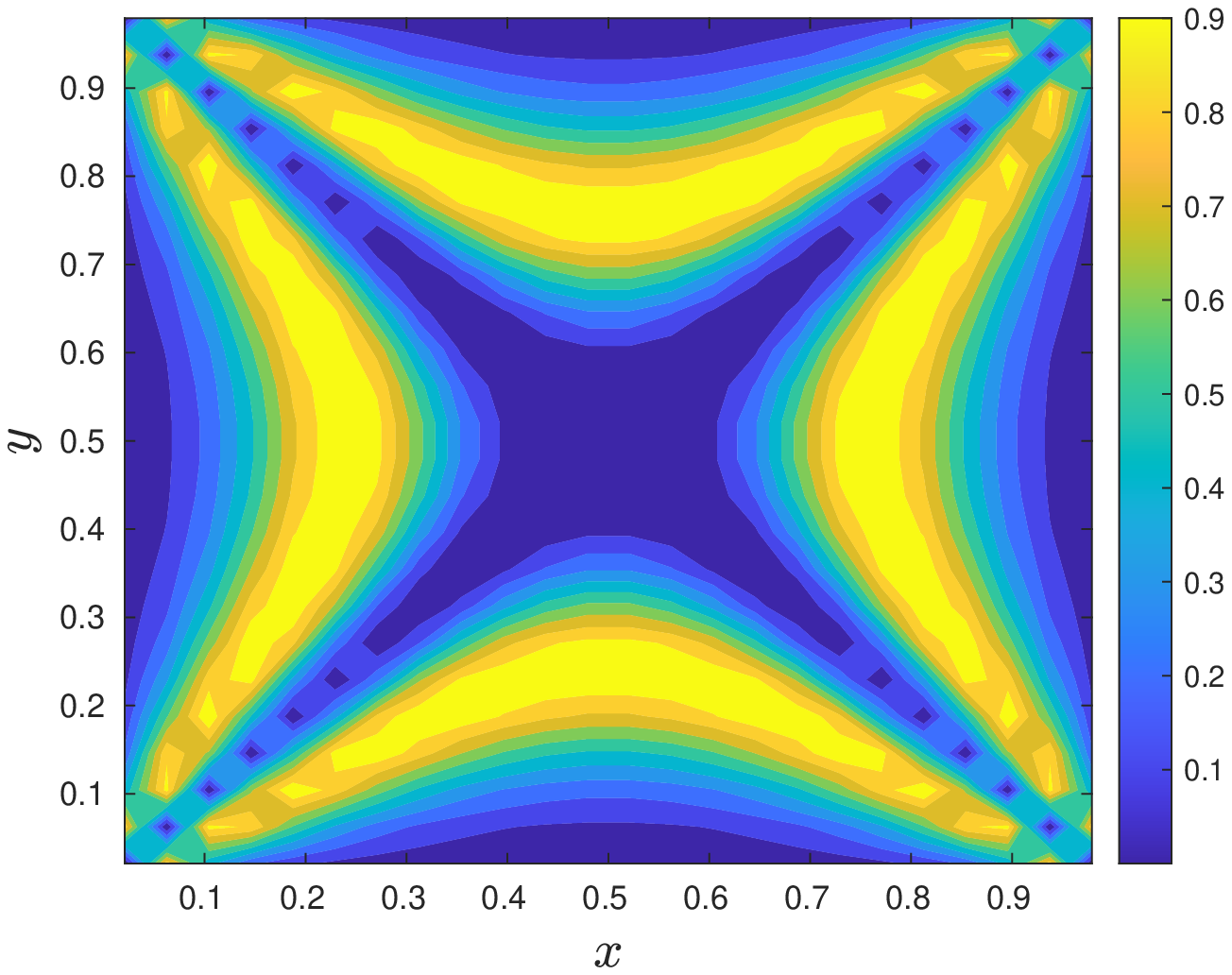}
   \label{fig:ex2_bia}}
  \caption{(Section \ref{Large-c02}) The steady state pattern. 
  } 
  \label{fig:ex2_eig_bia}
\end{figure}

Let us look into other features of our scheme using this example. 
For the energy dissipation, we plot the energy curve versus the time in Figure \ref{fig:ex2_tot_energy}, where we can see that the total energy is decreasing to a steady state. 
For the efficiency of the scheme, let us investigate the number of Newton's iteration, which is given in Figure \ref{fig:ex2_iter}. 
The maximum number of Newton's iteration is six, while for most time steps the number is no greater than four, indicating that our scheme can be implemented very efficiently.

\begin{figure}[!htb]
  \centering
   \subfloat[Total energy]{\includegraphics[width=0.45\textwidth,height=0.3\textwidth]{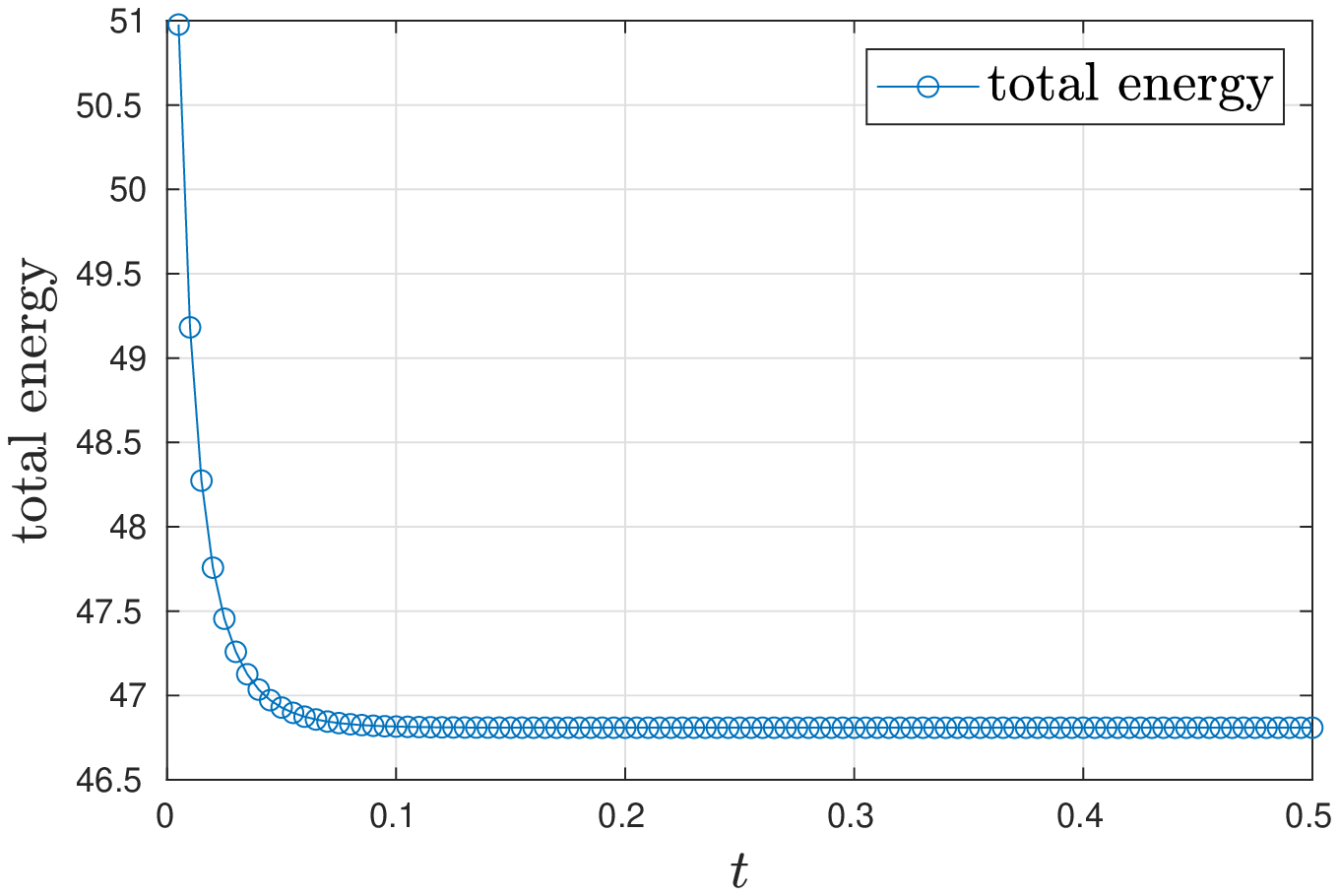}
   \label{fig:ex2_tot_energy}
   }\hfill
  \subfloat[Iteration number]{\includegraphics[width=0.45\textwidth,height=0.3\textwidth]{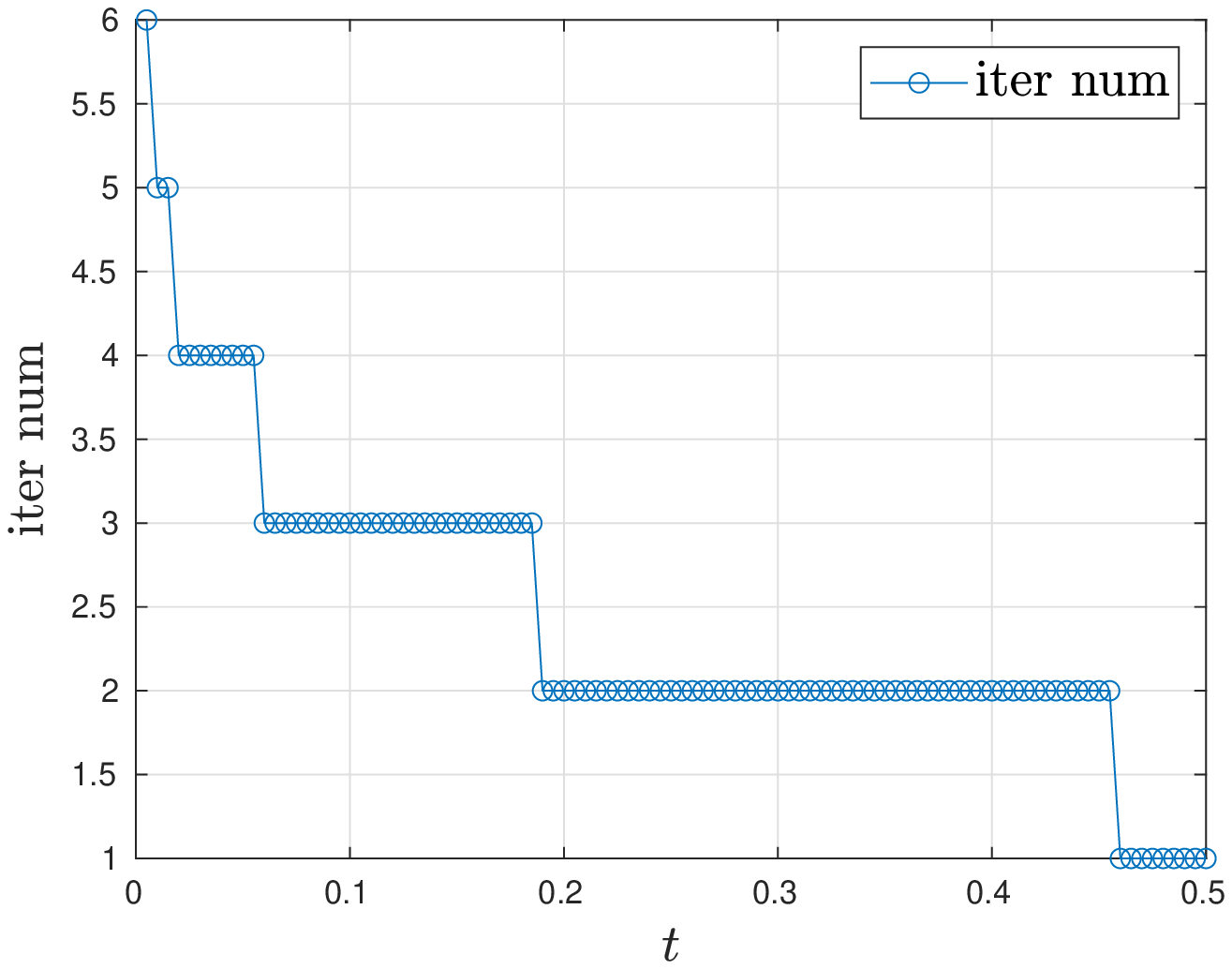}
  \label{fig:ex2_iter}
  }
  \caption{(Section \ref{Large-c02}) Total energy and the number of Newton's iterations versus the time. } 
  \label{fig:ex2_sol}
\end{figure}

\subsection{Effect of elastic coefficients}\label{Elasc}
We fix $c_{02} = 20$, $c_{21}=0.04$ and change the value of $c_{22}$. 
The boundary conditions are identical to those in Section \ref{Large-c02}. 
The initial condition is now chosen in the form \eqref{eq:ini} where $\bm{n}(\bm{x})$ takes a constant vector $(1/\sqrt{2},1/\sqrt{2},0)^t$. 
For the discretization, we still use the second-order scheme with $N = 24$ and $\delta t = 0.005$. 

We let the system evolve to a time long enough to reach a steady state. For six different values of $c_{22}$, we draw the principal eigenvectors in Figure \ref{fig:ex3_c21_vec}, and biaxiality in Figure \ref{fig:ex3_c21_bia}. 
It is observed from the principal eigenvectors that the pattern turns out to be the diagonal state \cite{yin2020construction}. 
As the value of $c_{22}$ increases, the transition of the principal eigenvector from a boundary horizontal/vertical direction to the diagonal direction in the middle becomes smoother, and the region with evident biaxiality also occupies a larger area.

\begin{figure}[!htb]
  \centering
\subfloat[$c_{22} = -0.039$]{\includegraphics[width=0.3\textwidth,height=0.22\textwidth]{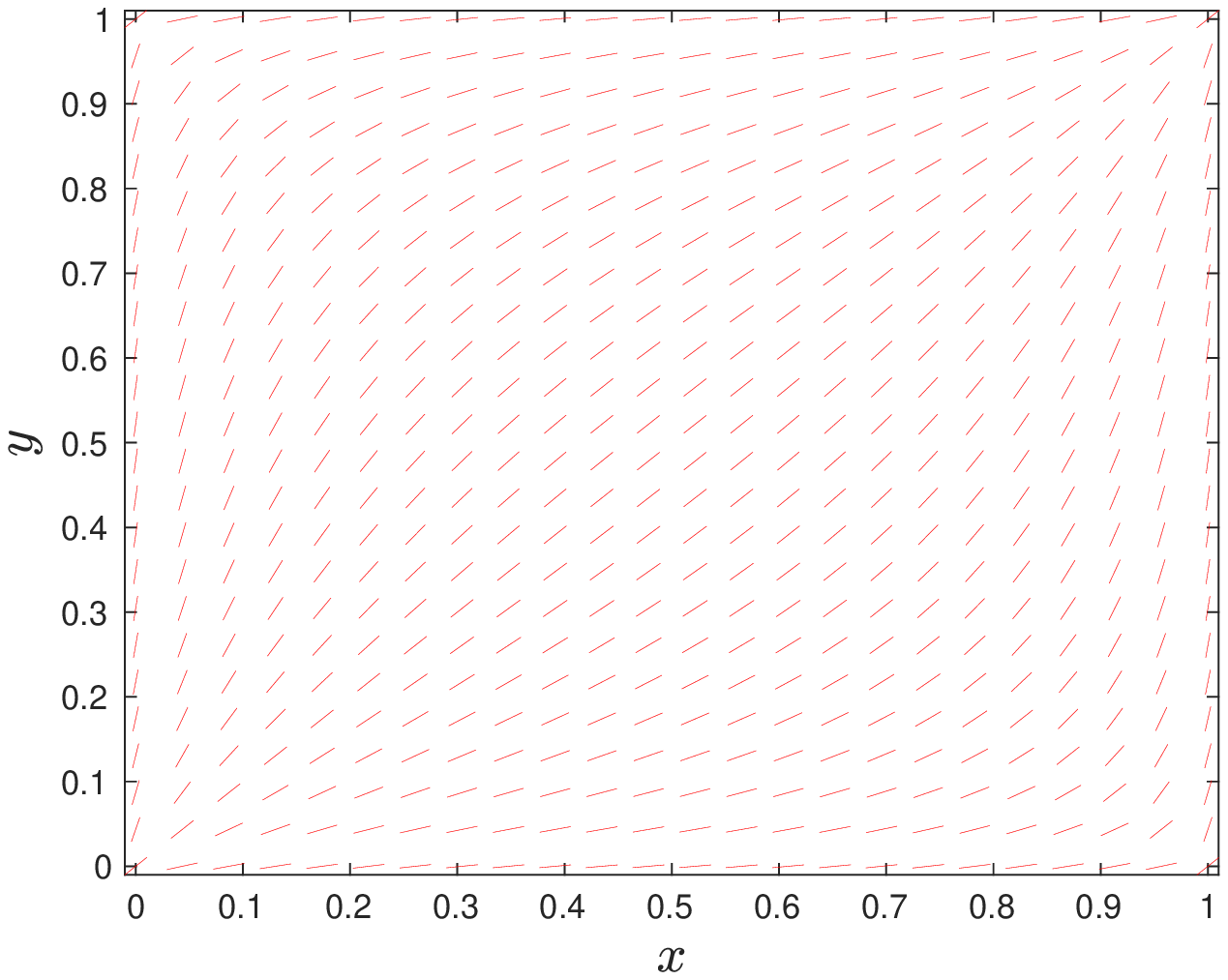}
   \label{fig:ex3_c21_vec_1}}\hfill
\subfloat[$c_{22} = -0.02$]{\includegraphics[width=0.3\textwidth,height=0.22\textwidth]{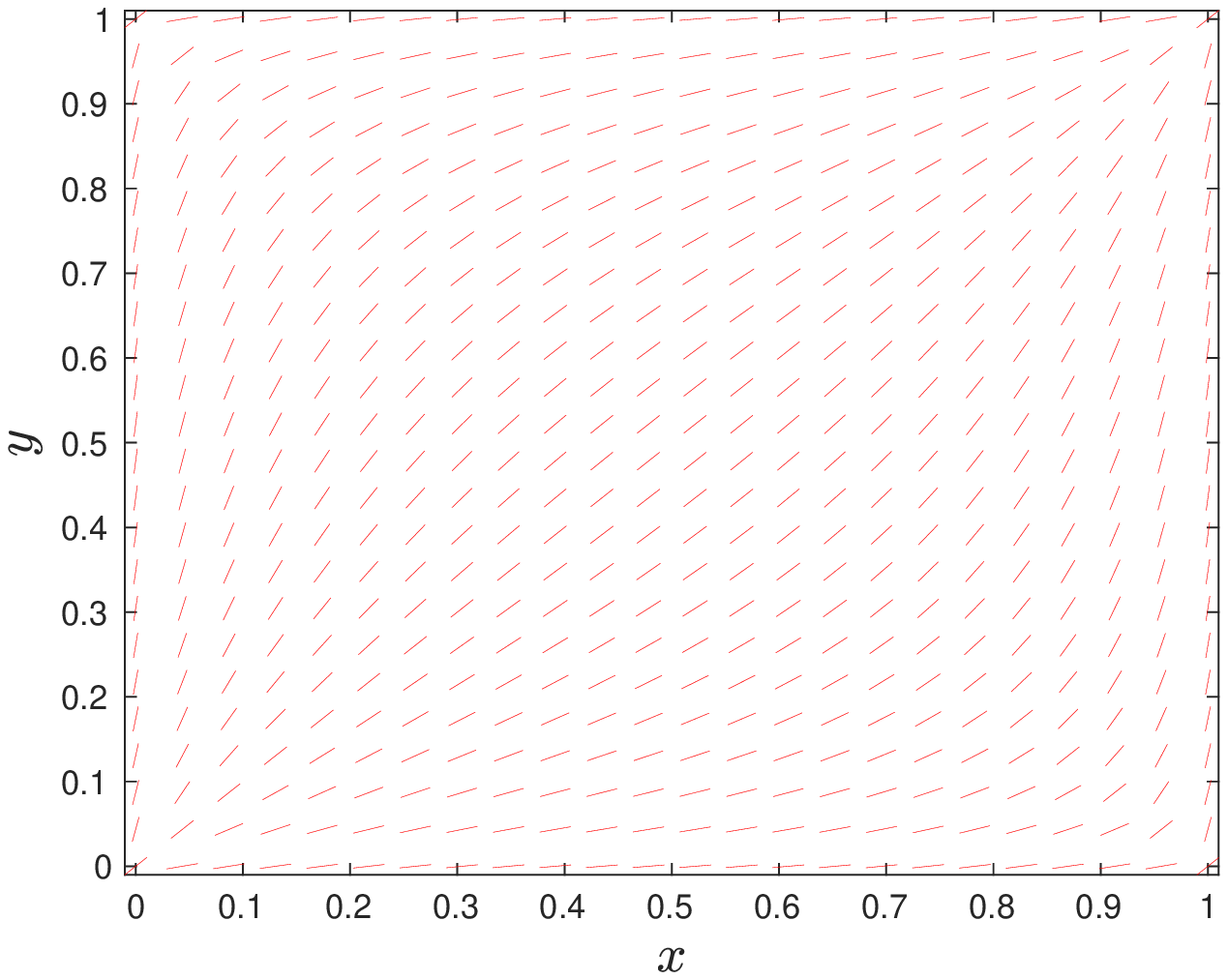}
   \label{fig:ex3_c21_vec_2}}\hfill
   \subfloat[$c_{22} = 0.0$]{\includegraphics[width=0.3\textwidth,height=0.22\textwidth]{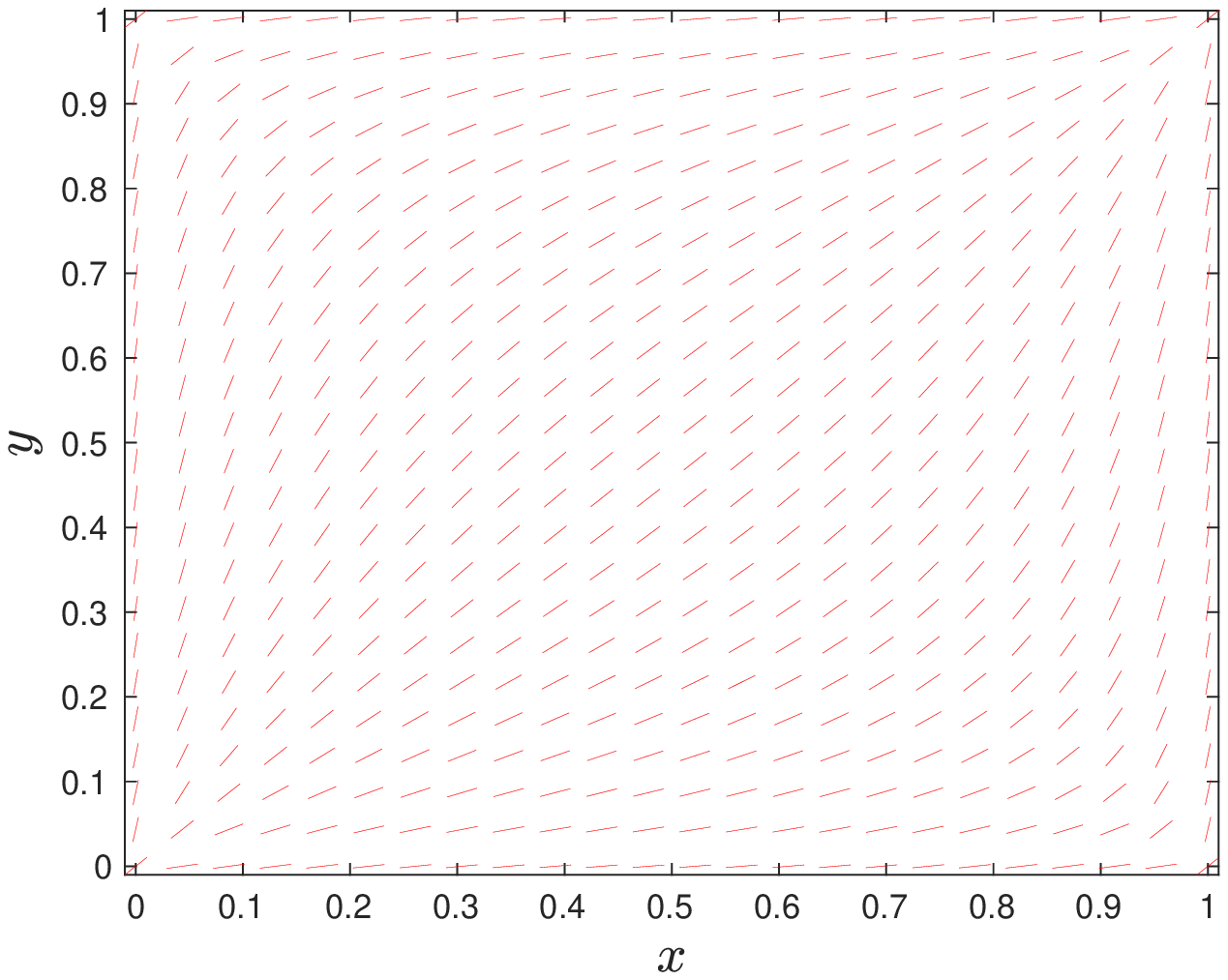}
   \label{fig:ex3_c21_vec_3}}\\
   \subfloat[$c_{22} = 0.04$]{\includegraphics[width=0.3\textwidth, height=0.22\textwidth]{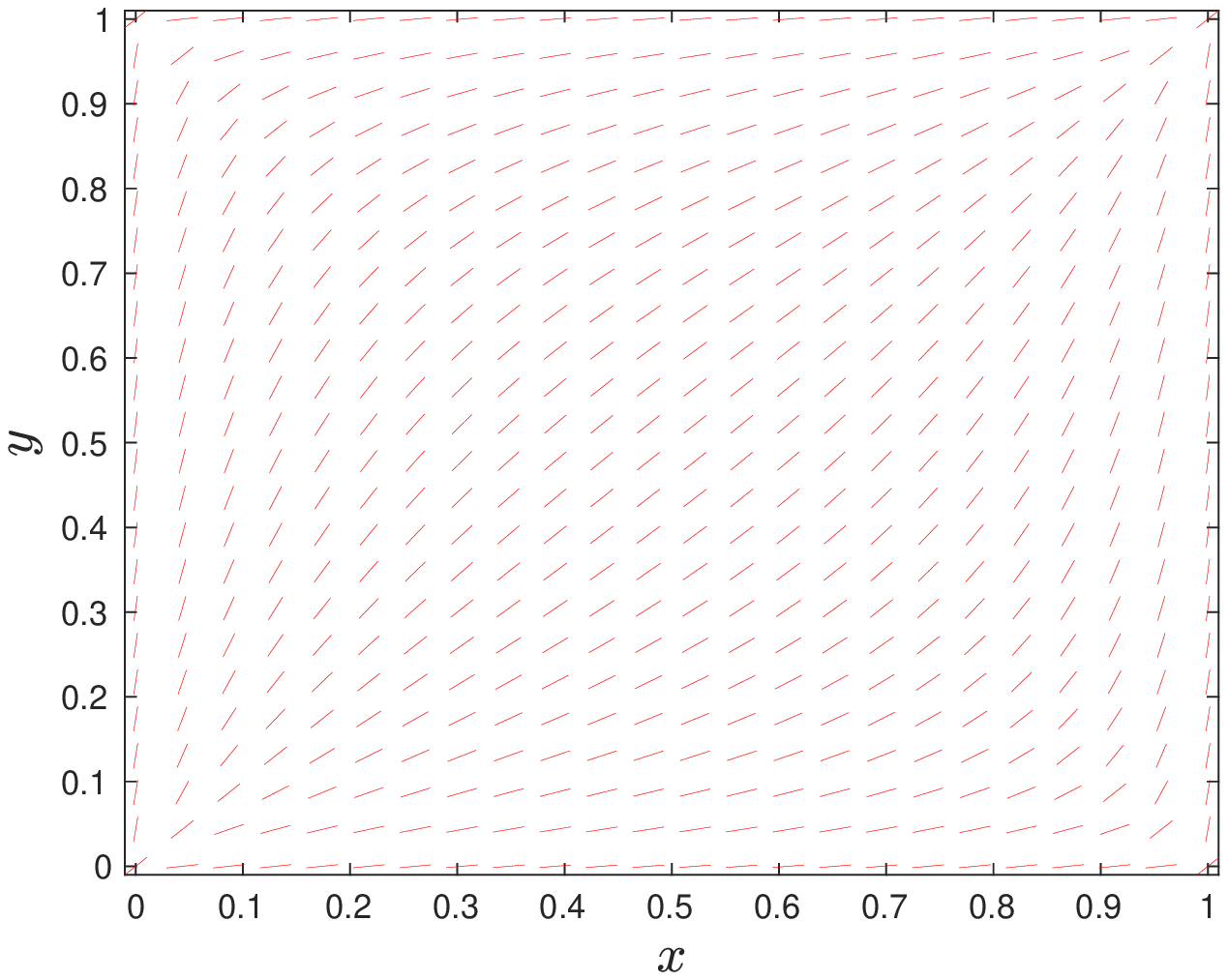}
   \label{fig:ex3_c21_vec_4}}\hfill
   \subfloat[$c_{22} = 0.16$]{\includegraphics[width=0.3\textwidth,height=0.22\textwidth]{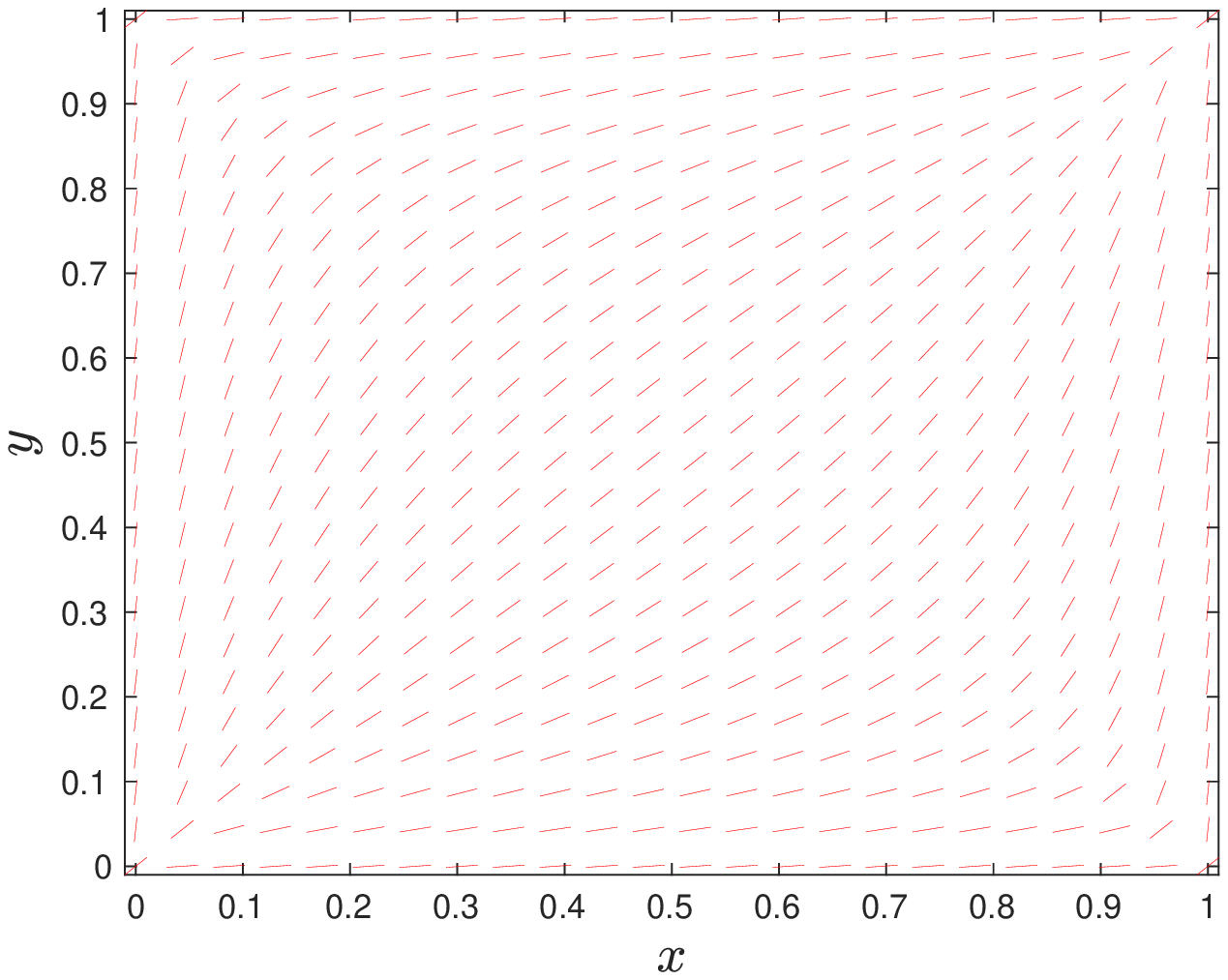}
   \label{fig:ex3_c21_vec_5}}\hfill
   \subfloat[$c_{22} = 0.32$]{\includegraphics[width=0.3\textwidth,height=0.22\textwidth]{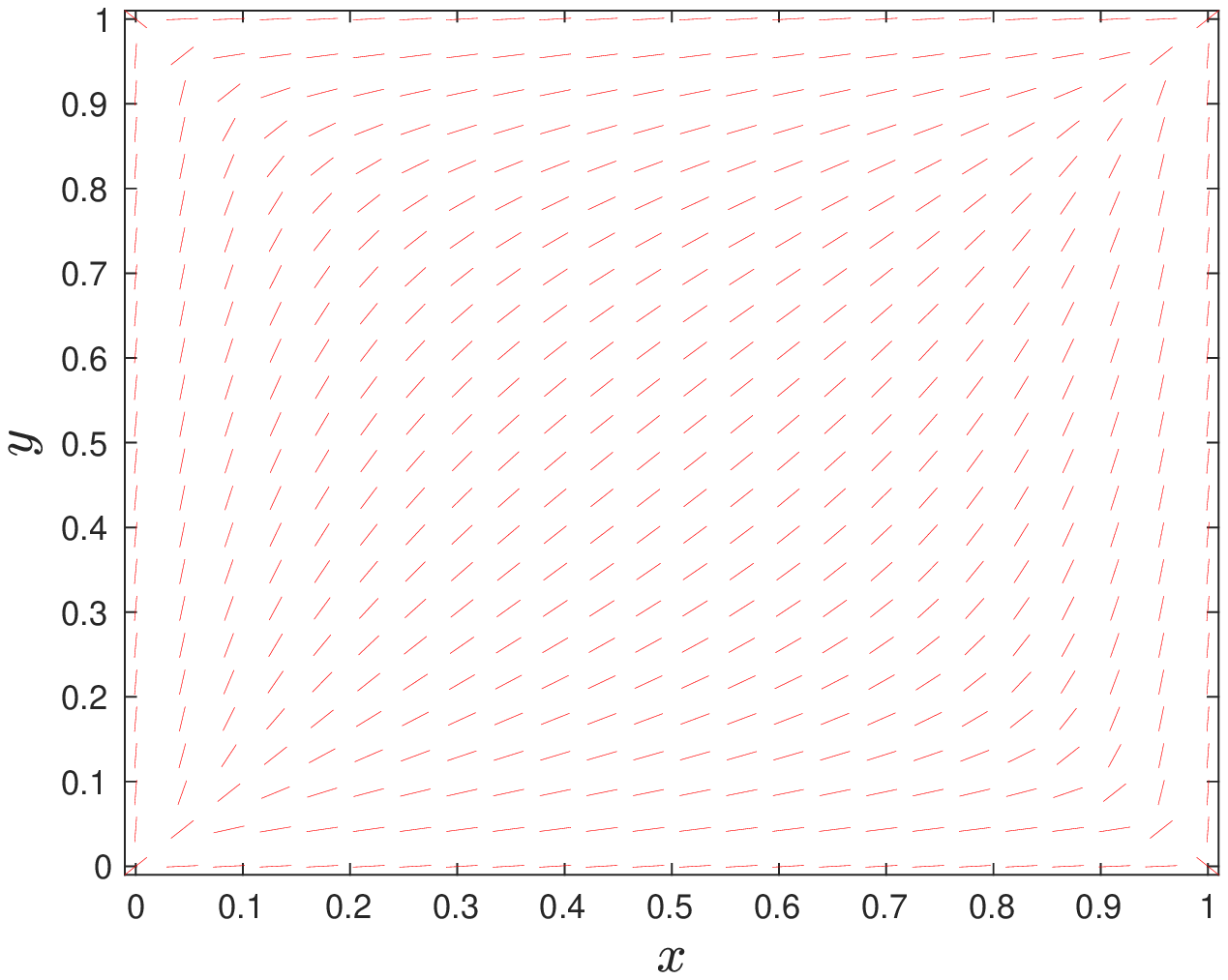}
   \label{fig:ex3_c21_vec_6}}
  \caption{(Section \ref{Elasc}) The principal eigenvector for different $c_{22}$ at the steady state. } 
  \label{fig:ex3_c21_vec}
\end{figure}

\begin{figure}[!htb]
  \centering
\subfloat[$c_{22} = -0.039$]{\includegraphics[width=0.3\textwidth,height=0.22\textwidth]{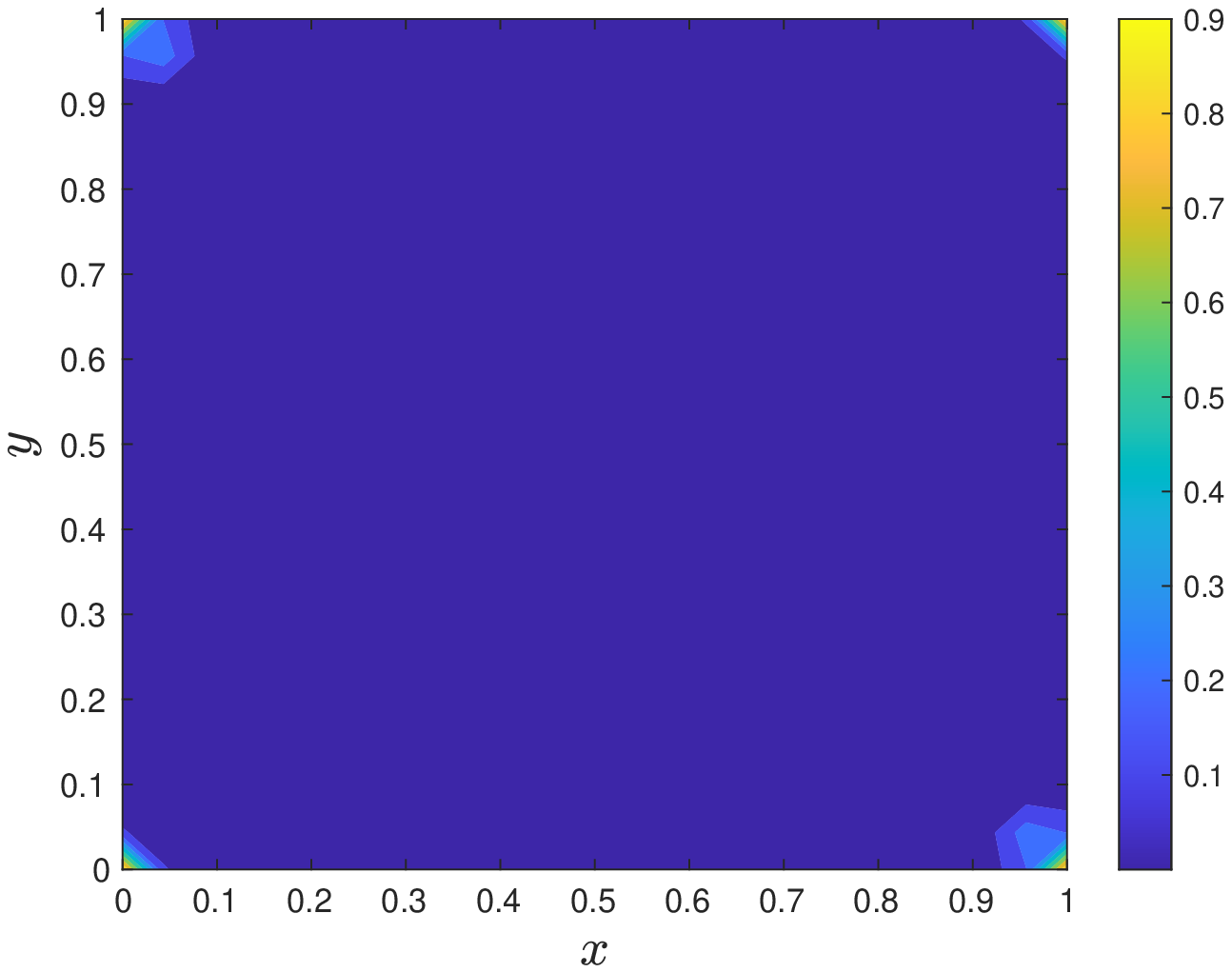}
   \label{fig:ex3_c21_11}}\hfill
\subfloat[$c_{22} = -0.02$]{\includegraphics[width=0.3\textwidth,height=0.22\textwidth]{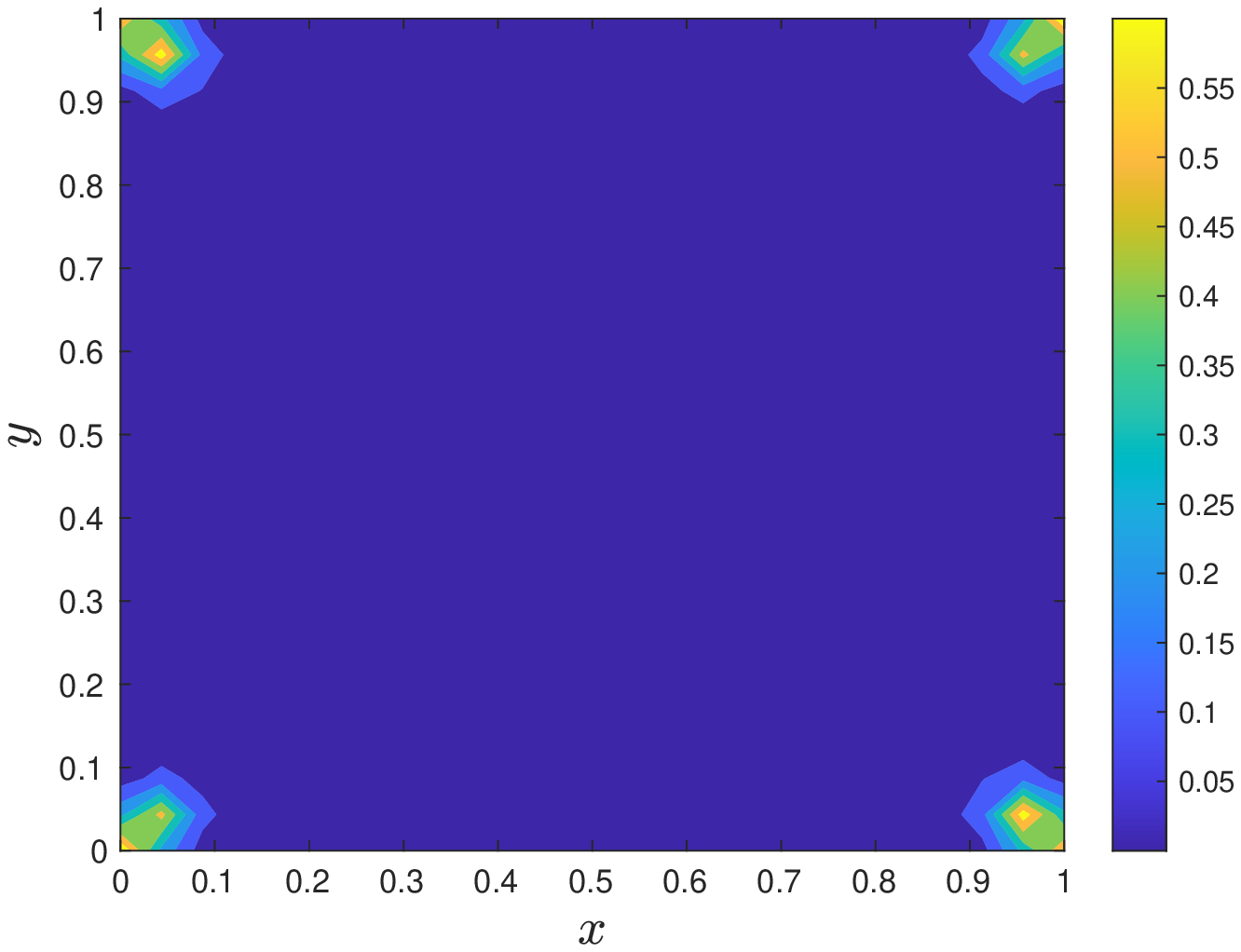}
   \label{fig:ex3_c21_12}}\hfill
   \subfloat[$c_{22} = 0.0$]{\includegraphics[width=0.3\textwidth,height=0.22\textwidth]{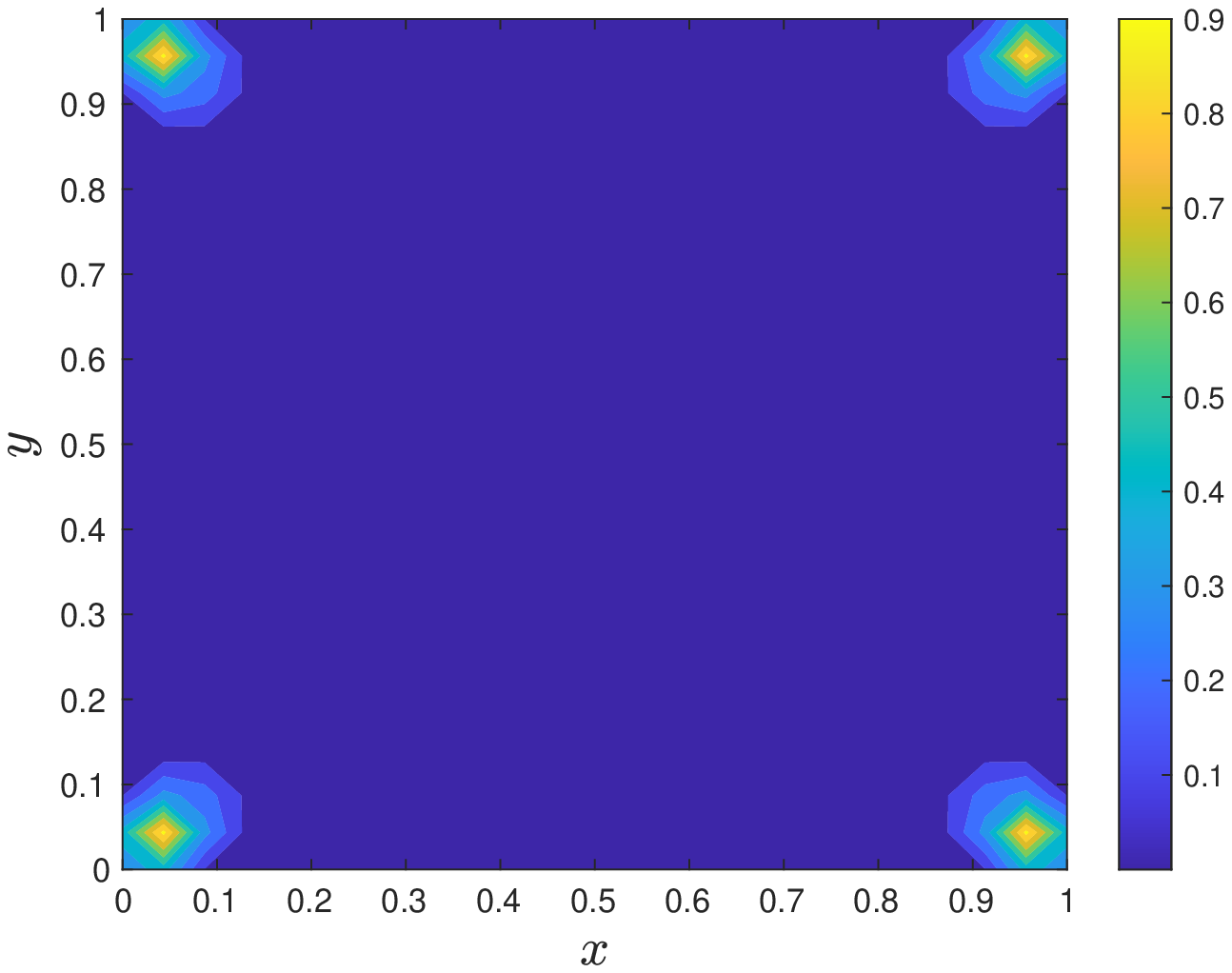}
   \label{fig:ex3_c21_13}}\\
   \subfloat[$c_{22} = 0.04$]{\includegraphics[width=0.3\textwidth,height=0.22\textwidth]{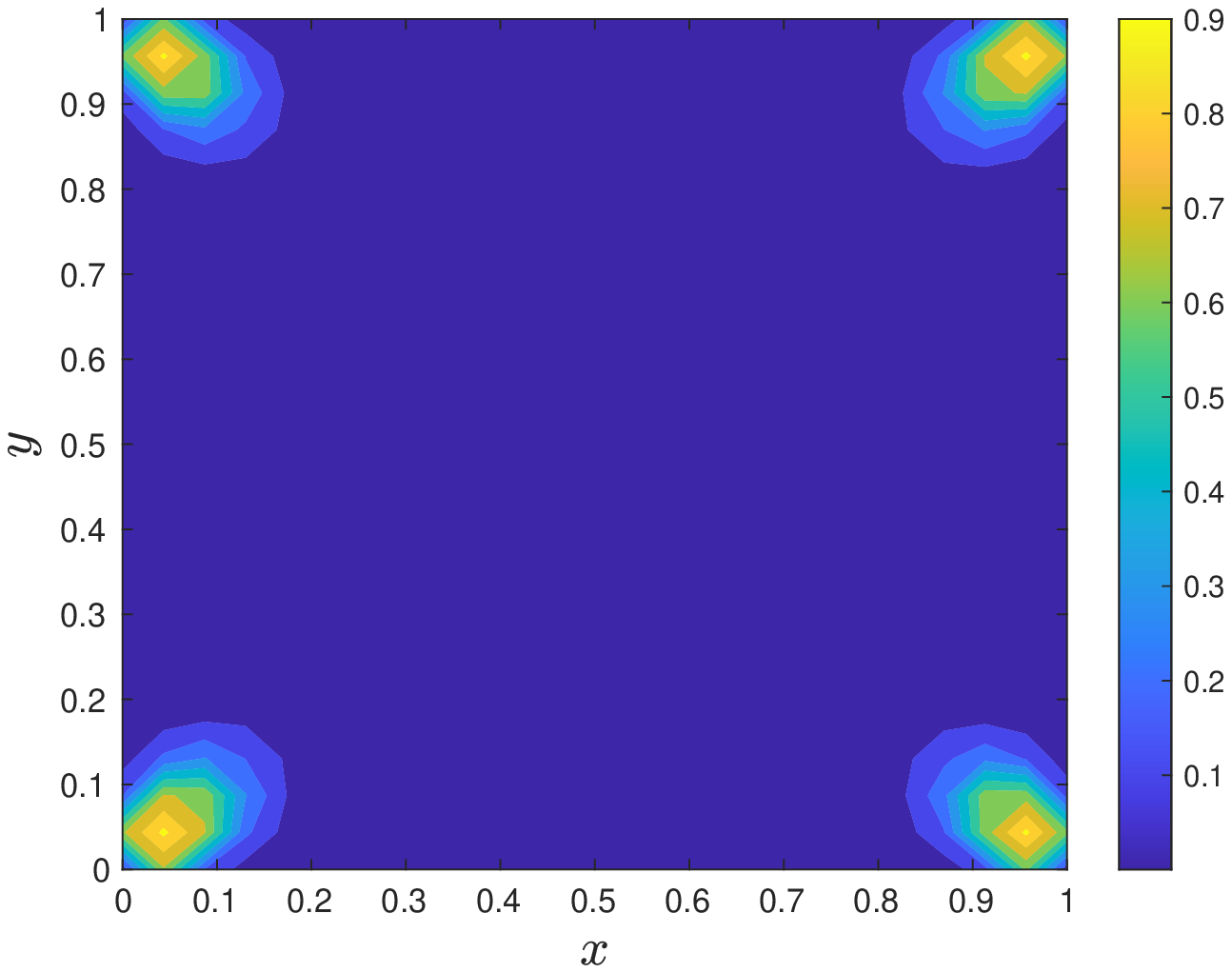}
   \label{fig:ex3_c21_14}}\hfill
   \subfloat[$c_{22} = 0.16$]{\includegraphics[width=0.3\textwidth,height=0.22\textwidth]{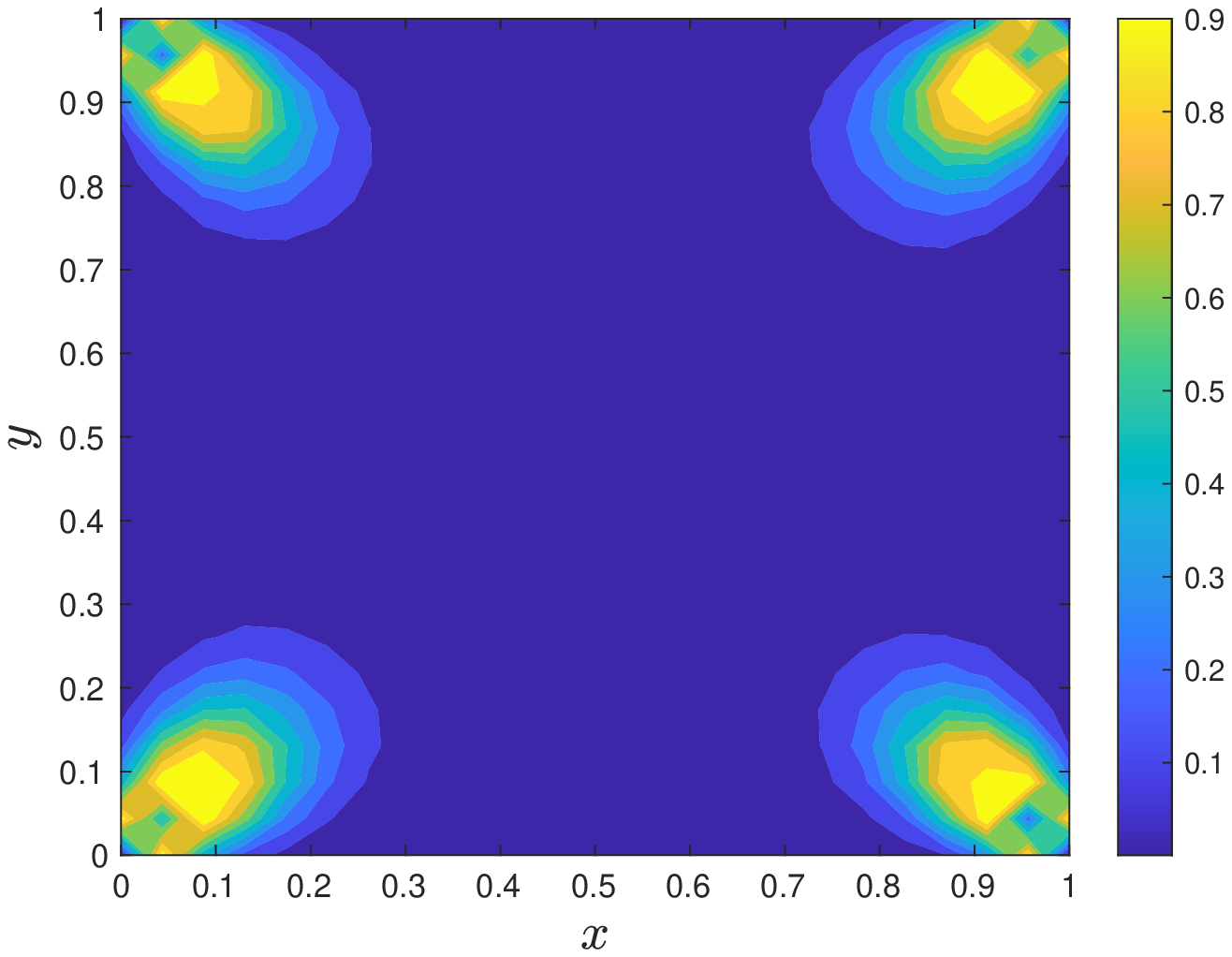}
   \label{fig:ex3_c21_15}}\hfill
   \subfloat[$c_{22} = 0.32$]{\includegraphics[width=0.3\textwidth,height=0.22\textwidth]{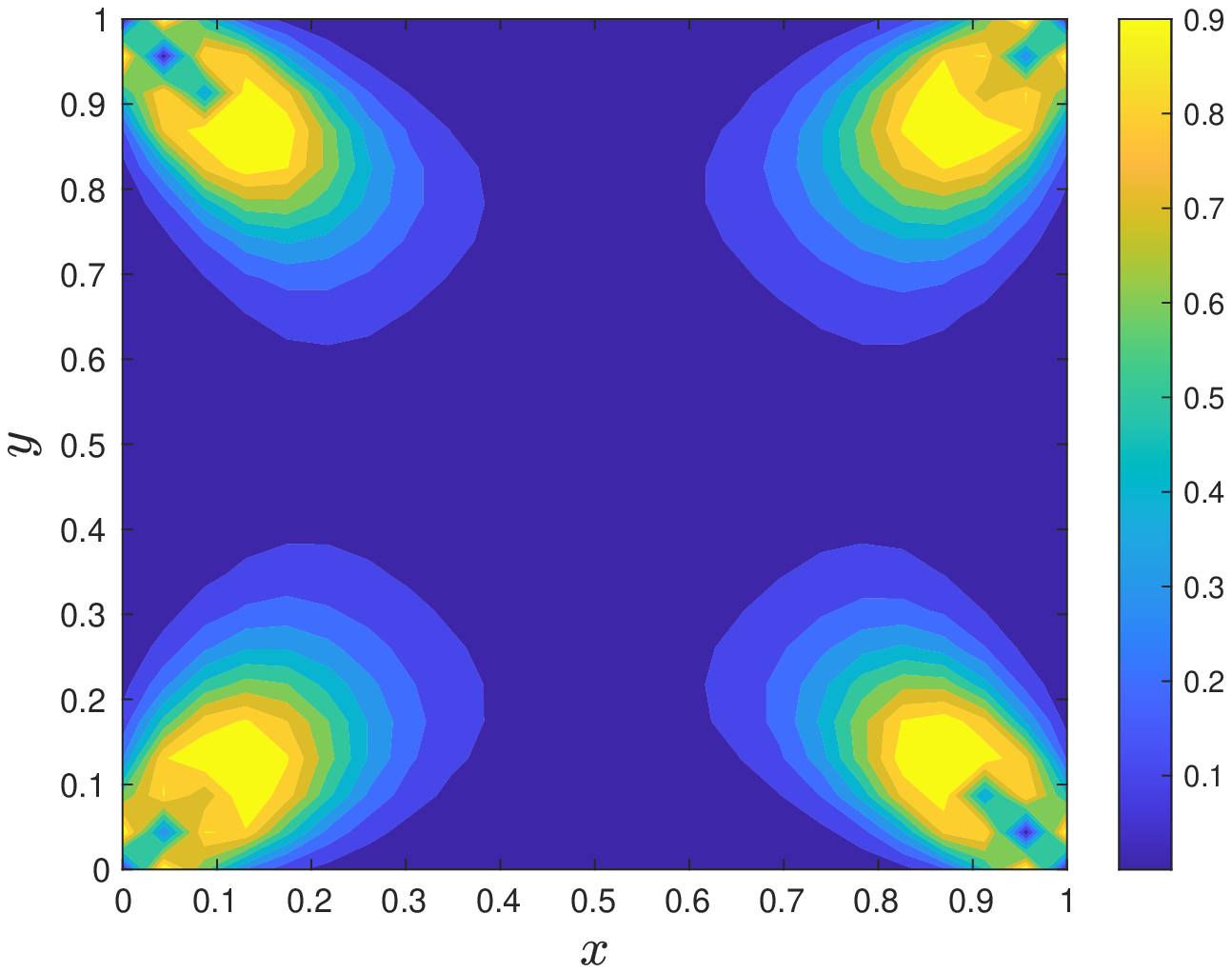}
   \label{fig:ex3_c21_16}}
  \caption{(Section \ref{Elasc}) The biaxiality for different $c_{22}$ at the steady state. } 
  \label{fig:ex3_c21_bia}
\end{figure}

\section{Conclusion}
\label{sec:conclusion}

We discuss the discretization for a gradient flow about $Q$-tensor with the quasi-entropy. 
The properties of the quasi-entropy are presented and utilized for designing the numerical schemes. 
The resulting schemes unconditionally keep physical constraints and are unconditionally uniquely solvable, for both semi discretizations in time and full discretizations. 
For the energy dissipation, the first-order-in-time schemes satisfy it unconditionally, while second-order-in-time schemes require an $O(1)$ restriction on the time step. 
Error estimates for these schemes are established. 
We validate the theoretical results with several numerical examples.

In the future, we expect to apply the approach in this paper to dissipative systems with high-order elastic energy terms and complex dissipative operators. 
They typically involve high-order tensors (such as \cite{iyer2015dynamic, Yu2010}), which have also attracted much attention. 
The approach in this paper is also suitable for the case where multiple tensors are involved and complex relations need to be met between them.

\appendix
\section{Bingham distribution and its properties}
\label{app:bingham}
We briefly introduce the Bingham term mentioned in the main text. 
Detailed discussions could be found in the literature (see, for example, \cite{Ball2010}). 

For $Q\in\mathcal{Q}_{\rm phys}$, consider the minimization problem 
\begin{align*}
    &\min \int_{S^2}\rho(\bm{m})\ln\rho(\bm{m})\dd\bm{m},\\
    &{\rm s.t. }\quad\int_{S^2}\rho(\bm{m})\dd\bm{m}=1,\quad \int_{S^2}\left(\bm{m}\otimes\bm{m}-\frac{1}{3}I\right)\rho(\bm{m})\dd\bm{m}=Q. 
\end{align*}
This problem has a unique solution of the form 
\begin{align}
    \rho(\bm{m})=\frac{1}{Z}\exp\Big(B(Q)\cdot(\bm{m}\otimes\bm{m}-\frac{1}{3}I)\Big),\qquad Z=\int_{S^2}\exp\Big(B(Q)\cdot(\bm{m}\otimes\bm{m}-\frac{1}{3}I)\Big)\dd\bm{m}, \label{bingham-dis}
\end{align}
where $B(Q)$ is a symmetric traceless $3\times 3$ matrix uniquely determined by $Q$. 
Such a density function is the Bingham distribution. 
The Bingham term in the free energy is given by taking \eqref{bingham-dis} into $\int_{S^2}\rho\ln\rho \dd \bm{m}$, giving 
\begin{align}
    \psi(Q)=B(Q)\cdot Q-\ln Z. \label{orig-ent}
\end{align}

It is not difficult to show that $\psi(Q)$ is rotationally invariant and strictly convex. 
As a result, we could focus on the case where $Q$ is diagonal, which implies that $B(Q)$ is also diagonal. 
Let us denote $Q=\mathrm{diag}(\lambda_1,\lambda_2,\lambda_3)$ and $B(Q)=\mathrm{diag}(\mu_1,\mu_2,\mu_3)$. 
If we assume $\lambda_1\geqslant \lambda_2\geqslant \lambda_3$, then it holds $\mu_1\geqslant\mu_2\geqslant\mu_3$. 
The asymptotic behavior of $\min\lambda(Q)\to (-1/3)^+$ is exactly that of $\mu_1-\mu_3\to +\infty$. 
It is shown in \cite{Ball2010} that when $\mu_1-\mu_3\to +\infty$, 
\begin{align*}
    & \lambda_3+\frac 13=\frac{C_1}{\mu_1-\mu_3}+o\left(\frac{1}{\mu_1-\mu_3}\right),\\
    & \psi(Q)=C_2\ln(\mu_1-\mu_3)+o(\mu_1-\mu_3). 
\end{align*}
Therefore, $\psi(Q)\sim -C_3\ln(\lambda_3+1/3)$. 

It is noticed that to compute \eqref{orig-ent}, it is necessary to solve $B(Q)$ first. 
However, the existing numerical approaches \cite{Luo2018, Kent1987, Kume2013, Grosso2000, jiang2021efficient} are still not able to provide a fast and accurate evaluation of $B(Q)$. 

\section{Summation by parts}
\label{app:partial_sum}
We derive the first equality in \eqref{eq:partial_sum2}. 
Substituting \eqref{eq:spatial_xy} into \eqref{eq:partial_sum}, we deduce that that 
\begin{equation}
    \label{eq:partial_sum_1}
    \sum_{P} (u D_{12}v)_{P} = 
    -\frac{1}{4h^2} \sum_{l,m=1}^{N-1} u_{l,m} (-v_{l+1, m+1} + v_{l+1, m-1} + v_{l-1, m+1} - v_{l-1, m-1}). 
\end{equation}
Substituting \eqref{eq:spatial_x1y2} into \eqref{eq:partial_sum_1}, we derive that 
\begin{align}
    \label{eq:partial_sum_2}
    & \sum_{\Gamma}(D_1 u)_{\Gamma}(D_{2}v)_{\Gamma}\nonumber\\
     = & 
    \frac{1}{4h^2} \sum_{l=0}^{N-1}\sum_{m=0}^{N-1}(-u_{l,m} -u_{l,m+1} + u_{l+1, m} + u_{l+1, m+1}) (-v_{l,m}+v_{l,m+1}-v_{l+1, m}+v_{l+1, m+1}) \nonumber\\
    =& \frac{1}{4h^2} \sum_{l=0}^{N-1}\sum_{m=0}^{N-1}u_{l,m}  (v_{l,m}-v_{l,m+1}+v_{l+1,m}-v_{l+1, m+1}) \nonumber\\
    &+\frac{1}{4h^2} \sum_{l=0}^{N-1}\sum_{m=1}^{N}u_{l,m}  (v_{l,m-1}-v_{l,m}+v_{l+1, m-1}-v_{l+1, m}) \nonumber\\
     & + \frac{1}{4h^2} \sum_{l=1}^{N}\sum_{m=0}^{N-1}u_{l,m}  (-v_{l-1,m}+v_{l-1,m+1}-v_{l, m}+v_{l, m+1}) \nonumber\\
      &+  \frac{1}{4h^2} \sum_{l=1}^{N}\sum_{m=1}^{N}u_{l,m}  (-v_{l-1,m-1}+v_{l-1,m}-v_{l, m-1}+v_{l, m})  \nonumber\\
      = &  \frac{1}{4h^2} \sum_{l=1}^{N-1}\sum_{m=1}^{N-1}u_{l,m}  (-v_{l+1, m+1} + v_{l+1, m-1} + v_{l-1, m+1} - v_{l-1, m-1})\nonumber\\
       & + \frac{1}{4h^2}\left( F_1(u,v)_{0,0} + \sum_{m=1}^{N-1}F_1(u,v)_{0,m} + \sum_{l=1}^{N-1}F_1(u,v)_{l,0}\right)  \nonumber\\ 
       & +  \frac{1}{4h^2}\left( F_2(u,v)_{0,N} + \sum_{m=1}^{N-1}F_2(u,v)_{0,m} + \sum_{l=1}^{N-1}F_2(u,v)_{l,N}\right) \nonumber\\
       & +  \frac{1}{4h^2}\left( F_3(u,v)_{N,0} + \sum_{m=1}^{N-1}F_3(u,v)_{N,m} + \sum_{l=1}^{N-1}F_3(u,v)_{l,0}\right) \nonumber\\
       & + \frac{1}{4h^2}\left( F_4(u,v)_{N,N} + \sum_{m=1}^{N-1}F_4(u,v)_{N,m} + \sum_{l=1}^{N-1}F_4(u,v)_{l,N}\right),
\end{align}
with 
\begin{equation}
    \label{eq:F_i}
    \begin{aligned}
    F_1(u, v)_{l,m} & = u_{l,m}(v_{l,m}-v_{l,m+1}+v_{l+1, m}-v_{l+1, m+1}),\\
    F_2(u, v)_{l,m} & = u_{l,m}(v_{l,m-1}-v_{l,m}+v_{l+1, m-1}-v_{l+1, m}), \\
    F_3(u, v)_{l,m} & = u_{l,m}(-v_{l-1,m}+v_{l-1,m+1}-v_{l, m}+v_{l, m+1}), \\
    F_4(u, v)_{l, m} & = u_{l,m}(-v_{l-1,m-1}+v_{l-1,m}-v_{l, m-1}+v_{l, m}). 
    \end{aligned}
\end{equation}
In each $F_i(u,v)$ in \eqref{eq:partial_sum_2}, the indices are located on the boundary. 
Since $u$ is zero on boundary nodes, the first equality in \eqref{eq:partial_sum2} has already been established.

\bibliographystyle{plain}
\bibliography{article}

\begin{thebibliography}{10}

\bibitem{Adler2015}
J.~Adler, T.~Atherton, D.~Emerson, and S.~MacLachlan.
\newblock An energy minimization finite-element approach for the
  {F}rank–{O}seen model of nematic liquid crystals.
\newblock {\em SIAM J. Numer. Anal.}, 53(5):2226--2254, 2015.

\bibitem{Alouges1997}
F.~Alouges.
\newblock A new algorithm for computing liquid crystal stable configurations:
  the harmonic mapping case.
\newblock {\em SIAM J. Numer. Anal.}, 34(5):1708--1726, 1997.

\bibitem{ambrosio2008gradient}
L.~Ambrosio, N.~Gigli, and G.~Savar{\'e}.
\newblock {\em Gradient flows: in metric spaces and in the space of probability
  measures}.
\newblock Springer Science \& Business Media, 2008.

\bibitem{Ball2010}
J.~Ball and A.~Majumdar.
\newblock Nematic liquid crystals: from {M}aier–{S}aupe to a continuum
  theory.
\newblock {\em Mol. Cryst. Liq. Cryst.}, 525(1):1--11, 2010.

\bibitem{Beris1994}
A.~Beris and B.~Edwards.
\newblock {\em Thermodynamics of flowing systems: with internal
  microstructure}.
\newblock Oxford University Press on Demand, 1994.

\bibitem{cai2017stable}
Y.~Cai, J.~Shen, and X.~Xu.
\newblock A stable scheme and its convergence analysis for a 2{D} dynamic
  {Q}-tensor model of nematic liquid crystals.
\newblock {\em Math. Models Methods Appl. Sci.}, 27(08):1459--1488, 2017.

\bibitem{Canevari2017}
G.~Canevari, A.~Majumdar, and A.~Spicer.
\newblock Order reconstruction for nematics on squares and hexagons: {A}
  {L}andau-de {G}ennes study.
\newblock {\em SIAM J. Appl. Math.}, 77(1):267--293, 2017.

\bibitem{Chen2019Pos}
W.~Chen, C.~Wang, X.~Wang, and S.~Wise.
\newblock Positivity-preserving, energy stable numerical schemes for the
  {C}ahn–{H}illiard equation with logarithmic potential.
\newblock {\em J. Comput. Phys. X}, 3:100031, 2019.

\bibitem{CHEN2016198}
Z.~Chen, H.~Huang, and J.~Yan.
\newblock Third order maximum-principle-satisfying direct discontinuous
  {G}alerkin methods for time dependent convection diffusion equations on
  unstructured triangular meshes.
\newblock {\em J. Comput. Phys.}, 308:198--217, 2016.

\bibitem{Cheng2021}
Q.~Cheng and J.~Shen.
\newblock A new {L}agrange multiplier approach for constructing structure
  preserving schemes, {I}. positivity preserving.
\newblock {\em arXiv preprint arXiv:2107.00504}, 2021.

\bibitem{Cohen1987}
R.~Cohen, R.~Hardt, D.~Kinderlehrer, S.~Lin, and M.~Luskin.
\newblock Minimum energy configurations for liquid crystals: {C}omputational
  results.
\newblock {\em Theory and Applications of Liquid Crystrals}, pages 99--121,
  1987.

\bibitem{deGennesProst1993}
P.~de~Gennes and J.~Prost.
\newblock {\em The physics of liquid crystals}, volume~83.
\newblock Oxford university press, 1993.

\bibitem{du2020maximum}
Q.~Du, L.~Ju, X.~Li, and Z.~Qiao.
\newblock Maximum bound principles for a class of semilinear parabolic
  equations and exponential time differencing schemes.
\newblock {\em SIAM Review}, 63(2):317--359, 2021.

\bibitem{Fatkullin2005}
I.~Fatkullin and V.~Slastikov.
\newblock Critical points of the {O}nsager functional on a sphere.
\newblock {\em Nonlinearity}, 18(6):2565, 2005.

\bibitem{Fukuda2004}
J.~Fukuda, H.~Stark, M.~Yoneya, and H.~Yokoyama.
\newblock Interaction between two spherical particles in a nematic liquid
  crystal.
\newblock {\em Phys. Rev. E}, 69(4):041706, 2004.

\bibitem{Golovaty2014}
D.~Golovaty and J.~Montero.
\newblock On minimizers of a {L}andau–de {G}ennes energy functional on planar
  domains.
\newblock {\em Arch. Ration. Mech. Anal.}, 213(2):447--490, 2014.

\bibitem{Grosso2000}
M.~Grosso, P.~Maffettone, and F.~Dupret.
\newblock A closure approximation for nematic liquid crystals based on the
  canonical distribution subspace theory.
\newblock {\em Rheol. Acta}, 39(3):301--310, 2000.

\bibitem{Han2014from}
J.~Han, Y.~Luo, W.~Wang, P.~Zhang, and Z.~Zhang.
\newblock From microscopic theory to macroscopic theory: a systematic study on
  modeling for liquid crystals.
\newblock {\em Arch. Ration. Mech. Anal.}, 215(3):741--809, 2015.

\bibitem{Hu2016}
Y.~Hu, Y.~Qu, and P.~Zhang.
\newblock On the disclination lines of nematic liquid crystals.
\newblock {\em Comm. Comput. Phys.}, 19(2):354--379, 2016.

\bibitem{HuaS21}
F.~Huang and J.~Shen.
\newblock Bound/positivity preserving and energy stable {SAV} schemes for
  dissipative systems: applications to {K}eller-{S}egel and
  {P}oisson-{N}ernst-{P}lanck equations.
\newblock {\em SIAM J. Sci. Comput.}, 43(3), 2021.

\bibitem{iyer2015dynamic}
G.~Iyer, X.~Xu, and A.~Zarnescu.
\newblock Dynamic cubic instability in a 2{D} {Q}-tensor model for liquid
  crystals.
\newblock {\em Math. Models and Methods Appl. Sci.}, 25(08):1477--1517, 2015.

\bibitem{jiang2021efficient}
S.~Jiang and H.~Yu.
\newblock Efficient spectral methods for quasi-equilibrium closure
  approximations of symmetric problems on unit circle and sphere.
\newblock {\em arXiv preprint arXiv:2104.14206}, 2021.

\bibitem{Kent1987}
J.~Kent.
\newblock Asymptotic expansions for the {B}ingham distribution.
\newblock {\em Applied statistics}, 36(2):139--144, 1987.

\bibitem{Kume2013}
A.~Kume, S.~Preston, and A.~Wood.
\newblock Saddlepoint approximations for the normalizing constant of
  {F}isher–{B}ingham distributions on products of spheres and {S}tiefel
  manifolds.
\newblock {\em Biometrika}, 100(4):971--984, 2013.

\bibitem{Kume2005}
A.~Kume and A.~Wood.
\newblock Saddlepoint approximations for the {B}ingham and {F}isher-{B}ingham
  normalising constants.
\newblock {\em Biometrika}, 92(2):465--476, 2005.

\bibitem{MR4186541}
B.~Li, J.~Yang, and Z.~Zhou.
\newblock Arbitrarily high-order exponential cut-off methods for preserving
  maximum principle of parabolic equations.
\newblock {\em SIAM J. Sci. Comput.}, 42(6):A3957--A3978, 2020.

\bibitem{Liu2005Axial}
H.~Liu, H.~Zhang, and P.~Zhang.
\newblock Axial symmetry and classification of stationary solutions of
  doi-{O}nsager equation on the sphere with {M}aier-{S}aupe potential.
\newblock {\em Commun. Math. Sci.}, 3(2):201--218, 2005.

\bibitem{liu2018positivity}
J.~Liu, L.~Wang, and Z.~Zhou.
\newblock Positivity-preserving and asymptotic preserving method for 2{D}
  {K}eller-{S}egal equations.
\newblock {\em Math. Comput.}, 87(311):1165--1189, 2018.

\bibitem{liu1996non}
X.~Liu and S.~Osher.
\newblock Nonoscillatory high order accurate self-similar maximum principle
  satisfying shock capturing schemes {I}.
\newblock {\em SIAM J. Numer. Anal.}, 33:760--779, 1996.

\bibitem{Liu2021Regularity}
Y.~Liu, X.~Lu, and X.~Xu.
\newblock Regularity of a gradient flow generated by the anisotropic
  {L}andau-de {G}ennes energy with a singular potential.
\newblock {\em arXiv preprint arXiv:2011.09541}, 2021.

\bibitem{MR3062022}
C.~Lu, W.~Huang, V.~Van, and S.~Erik.
\newblock The cutoff method for the numerical computation of nonnegative
  solutions of parabolic {PDE}s with application to anisotropic diffusion and
  lubrication-type equations.
\newblock {\em J. Comput. Phys.}, 242:24--36, 2013.

\bibitem{Luo2018}
Y.~Luo, J.~Xu, and P.~Zhang.
\newblock A fast algorithm for the moments of {B}ingham distribution.
\newblock {\em J. Sci. Comput.}, 75(3):1337--1350, 2018.

\bibitem{Song2015}
S.~Mei and P.~Zhang.
\newblock On a molecular based {Q}-tensor model for liquid crystals with
  density variations.
\newblock {\em SIAM Multiscale Model. Simul.}, 13(3):977--1000, 2015.

\bibitem{Mkaddem2000}
S.~Mkaddem and E.~Gartland.
\newblock Fine structure of defects in radial nematic droplets.
\newblock {\em Phys. Rev. E}, 62(5):6694--6705, 2000.

\bibitem{Nguyen2013}
L.~Nguyen and A.~Zarnescu.
\newblock Refined approximation for minimizers of a {L}andau-de {G}ennes energy
  functional.
\newblock {\em Cal. Var. Partial Differ. Equ.}, 47(1-2):383--432, 2013.

\bibitem{Qian1998}
T.~Qian and P.~Sheng.
\newblock Generalized hydrodynamic equations for nematic liquid crystals.
\newblock {\em Phys. Rev. E}, 58(6):7475, 1998.

\bibitem{Quarter2008}
A.~Quarteroni and A.~Valli.
\newblock {\em Numerical Approximation of Partial Differential Equations}.
\newblock Springer, 2008.

\bibitem{Ravnik2009}
M.~Ravnik and S.~\v{Z}umer.
\newblock Landau–de {G}ennes modelling of nematic liquid crystal colloids.
\newblock {\em Liq. Cryst.}, 36(10-11):1201--1214, 2009.

\bibitem{ShenXu2020}
J.~Shen and J.~Xu.
\newblock Unconditionally bound preserving and energy dissipative schemes for a
  class of {K}eller-{S}egel equations.
\newblock {\em SIAM J. Num Anal}, 58(3):1674--1695, 2020.

\bibitem{ShenXu2021}
J.~Shen and J.~Xu.
\newblock Unconditionally positivity preserving and energy dissipative schemes
  for {P}oisson–{N}ernst–{P}lanck equations.
\newblock {\em Numer. Math.}, 148:671--697, 2021.

\bibitem{Wang2008Crucial}
H.~Wang, K.~Li, and P.~Zhang.
\newblock Crucial properties of the moment closure model {FENE}-{QE}.
\newblock {\em J. Nonnewton Fluid Mech.}, 150(2):80--92, 2008.

\bibitem{wang2002hydro}
Q.~Wang.
\newblock A hydrodynamic theory for solutions of nonhomogeneous nematic liquid
  crystalline polymers of different configurations.
\newblock {\em J. Chem. Phys}, 116(20):9102--9136, 2002.

\bibitem{Zhanglei2021}
W.~Wang, L.~Zhang, and P.~Zhang.
\newblock Modeling and computation of liquid crystals.
\newblock {\em arXiv preprint arXiv:2014.02250}, 2021.

\bibitem{Wang2018}
Y.~Wang, P.~Zhang, and J.~Chen.
\newblock Formation of three-dimensional colloidal crystals in a nematic liquid
  crystal.
\newblock {\em Soft matter}, 14(32):6756--6766, 2018.

\bibitem{Wu2019Dyan}
H.~Wu, X.~Xu, and A.~Zarnescu.
\newblock Dynamics and flow effects in the {B}eris-{E}dwards system modeling
  nematic liquid crystals.
\newblock {\em Arch. Rational Mech. Anal.}, 231:1217--1267, 2019.

\bibitem{Xu2020quasi}
J.~Xu.
\newblock Quasi-entropy by log-determinant covariance matrix and application to
  liquid crystals.
\newblock {\em arXiv preprint arXiv:2007.15786}, 2020.

\bibitem{yin2020construction}
J.~Yin, Y.~Wang, J.~Chen, P.~Zhang, and L.~Zhang.
\newblock Construction of a pathway map on a complicated energy landscape.
\newblock {\em Phys. Rev. Lett.}, 124(9):090601, 2020.

\bibitem{Yu2010}
H.~Yu, G.~Ji, and P.~Zhang.
\newblock A nonhomogeneous kinetic model of liquid crystal polymers and its
  thermodynamic closure approximation.
\newblock {\em Comm. Comput. Phys.}, 7(2):383--4--2, 2020.

\bibitem{Yu2007}
H.~Yu and P.~Zhang.
\newblock A kinetic-hydrodynamic simulation of microstructure of liquid crystal
  polymers in plane shear flow.
\newblock {\em J. Non-Newtonian Fluid Mech.}, 141:116--127, 2007.

\bibitem{Zhang2010Onmax}
X.~Zhang and C.~Shu.
\newblock On maximum-principle-satisfying high order schemes for scalar
  conservation laws.
\newblock {\em J. Comput. Phys.}, 229(9):3091--3120, 2010.

\bibitem{Zhang2010OnPos}
X.~Zhang and C.~Shu.
\newblock On positivity-preserving high order discontinuous {G}alerkin schemes
  for compressible {E}uler equations on rectangular meshes.
\newblock {\em J. Comput. Phys.}, 229(23):8918--8934, 2010.

\bibitem{zhao2017novel}
J.~Zhao, X.~Yang, Y.~Gong, and Q.~Wang.
\newblock A novel linear second order unconditionally energy stable scheme for
  a hydrodynamic {Q}-tensor model of liquid crystals.
\newblock {\em Comput. Methods Appl. Mech. Eng.}, 318:803--825, 2017.

\end{thebibliography}
\end{document}